\documentclass{amsproc}
\usepackage{amsmath,amsthm,amsfonts,amssymb,amscd,amsbsy}
\usepackage{array,colortbl}
\usepackage{mathtools}

\newtheorem{theorem}{Theorem}[section]
\newtheorem{lemma}[theorem]{Lemma}

\newtheorem{proposition}[theorem]{Proposition}
\newtheorem{corollary}[theorem]{Corollary}

\theoremstyle{definition}
\newtheorem{remark}[theorem]{Remark}
\newtheorem{definition}[theorem]{Definition}
\newtheorem{notation}[theorem]{Notation}
\newtheorem{example}[theorem]{Example}
\newtheorem{question}[theorem]{Question}

\newtheorem{fact}[theorem]{Fact}

\newcommand{\cn}{\sqrt{n}}
\numberwithin{equation}{section}

\begin{document}
\title{Almost Fra\"iss\'e Banach spaces}

\author{Valentin Ferenczi}
\address{Universidade de S\~ao Paulo, Instituto de Matem\'atica e Estat\'istica, São Paulo, Brasil (IME-USP)}
\email{ferenczi@ime.usp.br}

\author{Michael A. Rinc\'on-Villamizar}
\address{Universidade de S\~ao Paulo, Instituto de Matem\'atica e Estat\'istica, São Paulo, Brasil (IME-USP)\\
Universidad Industrial de Santander (UIS), Escuela de Matem\'aticas, Bucaramanga, Colombia}
\email{marinvil@alumni.usp.br, marinvil@uis.edu.co}

\thanks{The first author was supported by the S\~ao Paulo Research Foundation (FAPESP), grant 2016/25574-8, and by Conselho Nacional de Desenvolvimento Cient\'ifico e Tecnol\'ogico (CNPq), grant 303731/2019-2.. The second author was supported by the S\~ao Paulo Research Foundation (FAPESP), grants 2016/25574-8 and 2021/01144-2, and Universidad Industrial de Santander (UIS)}



\begin{abstract}

  Continuing with the study of Approximately ultrahomogeneous and Fra\"iss\'e Banach spaces introduced by V. Ferenczi, J. L\'opez-Abad, B. Mbombo and S. Todorcevic, we define  formally weaker and in some aspects more natural properties of Banach spaces which we call Almost ultrahomogeneity and the Almost Fra\"iss\'e Property. We obtain relations between these different homogeneity properties of a space $E$ and relate them to certain pseudometrics on the class $\mathrm{Age}(E)$ of finite dimensional subspaces of $E$. We prove that ultrapowers of an almost Fra\"iss\'e Banach space are ultrahomogeneous. We also study two properties called finitely isometrically extensible and almost finitely isometrically extensible, respectively, and prove that approximately ultrahomogeneous Banach spaces are finitely isometrically extensible.

Finally, we study $\omega$-categoricity in Banach spaces, and give a proof that a  Banach space is Fra\"iss\'e   if and only if it is approximately ultrahomogeneous and $\omega$-categorical. 
\end{abstract}


\maketitle

\tableofcontents

\section{Introduction}

 The standard terminology and notation of Banach spaces used in this paper may be found in \cite{AK}. Recall that a Banach space is called \textit{transitive} if for any two points on the unit sphere, there is an onto isometry defined on the whole space that sends one onto the other. In 1932, S. Mazur conjectured that separable transitive Banach spaces are Hilbert \cite[p. 151]{banach}. This is the Mazur rotation problem (two surveys on this problem and related topics are \cite{cabello-sanchez} and \cite{ferenczi-cabello-beata}). Of course Hilbert spaces are transitive, and it is known that there exist non-separable and non-hilbertian transitive Banach spaces. The paper outgrowths from the study of multidimensional aspects of the Mazur rotation problem, as initiated in \cite{ferenczi et all} and pursued in \cite{ferenczi-lopez}.

 The notation we shall use for Fra\"iss\'e theory in Banach spaces is from \cite{ferenczi et all} and is recalled here.  We denote by $\mathrm{Age}(E)$ and $\mathrm{Age}_k(E)$ the set of all finite-dimensional subspaces of a Banach space $E$ and the set of all $k$-dimensional subspaces of $E$, respectively. Given two Banach spaces $X$ and $E$, and $\delta \geq 0$, a $\delta$-isometry from $X$ to $E$ is a linear map $T\colon X\to E$ such that $\frac{1}{1+\delta}\|x\|\leq\|Tx\|\leq(1+\delta)\|x\|$ for all $x\in X$; $\mathrm{Emb}(X,E)$ ($\mathrm{Emb}_\delta(X,E)$) is the (possibly empty) set of isometries ($\delta$-isometries, resp.) from $X$ into $E$. Also, $\mathrm{Isom}(E)$ ($\mathrm{Isom}_{\varepsilon}(E)$) denotes the set of all surjective isometries ($\varepsilon$-isometries) on $E$.

\subsection{Ultrahomogeneity, Approximately ultrahomogeneity, FIE-ness and almost FIE-ness}

An approach towards the Mazur rotation problem is by considering properties which are stronger than transitivity, satisfied by Hilbert spaces and such that there are non-separable non-Hilbertian examples: if we can prove that the only separable space with such a property is the Hilbert space, then arguably this is a first step towards a positive answer to Mazur problem in which the separability hypothesis was used.  A fundamental example is the multidimensional version of transitivity, which is called \textit{ultrahomogeneity}, \cite[Definition 2.2]{ferenczi et all}. A Banach space $E$ is \textit{ultrahomogeneous} (\textbf{UH}) when for every $X\in\mathrm{Age}(E)$ every element of $\mathrm{Emb}(X,E)$ can be extended to an element of $\mathrm{Isom}(E)$, or equivalently the group $\mathrm{Isom}(E)$ acts transitively on the metric space $\mathrm{Emb}(X,E)$ for all $X\in\mathrm{Age}(E)$. Of course, Hilbert spaces are ultrahomogeneous. Also ultrapowers of the Gurarij space or of $L_p$-spaces are non-separable examples of ultrahomogeneous spaces for $p\not\in2\mathbb N+4$,
\cite[Proposition 4.13]{aviles et al} and \cite[Corollary 2.16]{ferenczi et all}.

 To go on, let us recall a weaker form of the notion of transitivity: namely the property of almost-transitivity. A Banach space is called \textit{almost transitive} if the orbits of the isometry group of the space are dense in the unit sphere.  A. Pe\l czy\'nsky and S. Rolewicz proved that the Banach space $L_p(0,1)$ is almost transitive \cite{pel-rolewicz} when $1 \leq p<+\infty$. W. Lusky gave a multidimensional version of this result by showing that the group $\mathrm{Isom}(L_p(0,1))$ acts almost transitively on each metric space $\mathrm{Emb}(X,L_p(0,1))$ whenever $X$ is a finite dimensional subspace of $L_p(0,1)$ and $p=2$ or $p\not\in 2\mathbb N$ \cite{lusky}. In the language introduced in \cite{ferenczi et all}, Lusky's result says that for those values of $p$, $L_p(0,1)$ is \textbf{AUH}: a Banach space $E$ is \textit{approximately ultrahomogeneous} (in short \textbf{AUH}) if for each $X\in\mathrm{Age}(E)$, the group $\mathrm{Isom}(E)$ acts almost transitively on $\mathrm{Emb}(X,E)$, that is, if for each $\varepsilon>0$ and $\phi,\psi\in\mathrm{Emb}(X,E)$, there exists $T\in\mathrm{Isom}(E)$ such that $\|T\phi-\psi\|<\varepsilon$. In Lusky's paper only a formally weaker extension property of those $L_p$ spaces  is stated, which we call here 
 \textit{almost ultrahomogeneity}: a Banach space $E$ is \textit{almost ultrahomogeneous} (almost \textbf{UH}) if for each $\varepsilon>0$, each $X\in\mathrm{Age}(E)$ and each $\phi\in\mathrm{Emb}(X,E)$, there exists $T\in\mathrm{Isom}_\varepsilon(E)$ such that $T|_X=\phi$.
Therefore almost \textbf{UH} is a natural multidimensional counterpart of almost transitivity. So, we have the following relationships:
\[
\begin{array}
[c]{cccccccc}
  \mbox{Ultrahomogeneous} \Rightarrow   & \textbf{AUH}  &\Rightarrow \mbox{Almost Ultrahomogeneous} \\
\Downarrow  &  \\
\mbox{Transitive} &  \Rightarrow &   \mbox{Almost transitive}\\
\end{array}
\]

Motivated by the three aforementioned properties, we start by introducing two properties which we call \textit{finitely isometrically extensible} (\textbf{FIE}) and almost \textit{finitely isometrically extensible} (\textbf{aFIE}). A Banach space $E$ is \textbf{FIE} if any isometry defined on a finite dimensional subspace of $E$ can be extended to a norm-one operator defined on the whole space.
The \textbf{aFIE}-property is an $\varepsilon$-version of the \textbf{FIE}-property: a Banach space $E$ is \textbf{aFIE} if for any $\varepsilon>0$, any isometry defined on a finite dimensional subspace of $E$ can be extended to an operator defined on the whole space with norm at most $1+\varepsilon$. Contrarily to the transitivity property, the \textbf{FIE} is preserved by taking $1$-complemented subspaces, and in particular the 1-dimensional part of the \textbf{FIE}-property always holds: any isometry defined on a 1-dimensional subspace of $E$ always can be extended to a one-norm operator defined on $E$. It is clear that ultrahomogeneous spaces are \textbf{FIE}, but also other spaces, such as 
  $c_0(\Gamma)$ for any non-empty set $\Gamma$, are \textbf{FIE}. Also, we prove that almost ultrahomogeneous reflexive spaces are \textbf{FIE}. We do not know if 
\textbf{aFIE} Banach spaces are in fact \textbf{FIE}, but we note that the two notions coincide in the case of reflexive Banach spaces.  
 
 \subsection{Fra\"iss\'eness, Almost Fra\"iss\'eness and $\omega$-categoricity of Banach spaces}
 V. Ferenczi, J. L\'opez-Abad, B. Mbombo and S. Todorcevic defined and studied Fra\"iss\'e Banach spaces, proving that $L_p(0,1)$ is Fra\"iss\'e when $p=2$ or $p\not\in 2\mathbb N$ \cite[Theorem 4.1]{ferenczi et all}.  According to \cite{ferenczi et all}, a Banach space $E$ is \textit{Fra\"iss\'e} if for every $\varepsilon>0$ and every dimension $k$, there exists $\delta>0$ such that if $X$ is a $k$-dimensional subspace of $E$ the action of $\mathrm{Isom}(E)$ on $\mathrm{Emb}_\delta(X,E)$ has dense orbits. Another example of Fra\"iss\'e Banach space is the Gurarij space $\mathbb G$.
The existence of separable Fra\"iss\'e Banach spaces other than $\mathbb G$ and $L_p(0,1)$ is a main open problem \cite[Problem 2.9]{ferenczi et all}. 
 
 As developed in \cite{ferenczi et all}, Fra\"iss\'e Banach spaces are extremely important in relation to certain Ramsey properties of the classes ${\rm Emb}(X,E)$ and to the extreme amenability of the topological group ${\rm Isom}(E)$. Here we focus on by the following isometric properties of Fra\"iss\'e spaces:
\begin{enumerate}
    \item In Fra\"iss\'e spaces $E$, the Banach-Mazur and a restricted version of the Kadets pseudometrics are uniformly equivalent on the set of finite dimensional subspaces of $E$;
    \item Separable Fra\"iss\'e spaces are isometrically determined,  among Fra\"iss\'e spaces, by their local structure;
    \item Separable spaces who are finitely representable in a Fra\"iss\'e space $E$ can be isometrically embedded into $E$;
    \item For every non-free ultrafilter $\mathcal U$ on $\mathbb N$, the ultrapower $E_\mathcal U$ of a Fra\"iss\'e space $E$ is \textbf{UH};
\end{enumerate}
See \cite[Theorems 2.12 and 2.19, Propositions 2.13 and 2.15]{ferenczi et all}.
It is  natural to ask whether the full strength of the Fra\"iss\'e property is needed for these results. 
Finding weaker conditions could possibly lead to interesting new properties as well as shed some new light on the Fra\"iss\'e property.
With this aim in mind,
we introduce in this paper the \textit{almost Fra\"iss\'e Banach spaces} (\textbf{AF} in short). A Banach space $E$ is \textbf{AF} if for every $\varepsilon>0$ and every dimension $k\in\mathbb N$, there is $\delta>0$ such that if $X\in\mathrm{Age}_k(E)$ and $\phi\in\mathrm{Emb}_\delta(X,E)$, there is $T\in\mathrm{Isom}_\varepsilon(E)$ which extends $\phi$. We also study a weakening of this property which we call here \textit{weak almost Fra\"iss\'e} (weak \textbf{AF} shortly). A Banach space is weak \textbf{AF} whenever $X \in \mathrm{Age}(E)$ and $\varepsilon>0$, there exists $\delta>0$ such that if $\phi\in{\rm Emb}_\delta(X,E)$, there exists $T \in {\rm Isom}_{\varepsilon}(E)$ such that
  $T|_{X}=\phi$. This variation imitates the notion of weak Fra\"iss\'e space from \cite{ferenczi et all}. As it is expected, every Fra\"iss\'e space is \textbf{AF}, every weak Fra\"iss\'e space is weak \textbf{AF} and \textbf{AF} spaces are weak \textbf{AF}. When $\mathrm{Age}_k(E)$ is compact for each $k$ (with respect to the Banach-Mazur distance), the two above class introduced coincides, as we prove below.
This is similar to \cite[Theorem 2.12]{ferenczi et all}, where it was proved that the Fra\"iss\'e and weak Fra\"iss\'e properties coincide if $\mathrm{Age}_k(E)$ is compact for each $k$.
The almost and weak almost Fra\"iss\'e properties only involve $\varepsilon$-isometries, and in this sense, may seem more natural than their Fra\"iss\'e counterparts, where the definition envolves a mixture of isometries with $\varepsilon$-isometries.
We highlight the following facts about these classes of spaces (See Subsections \ref{subsections pseudodistances} and \ref{ultrapowers}).

\begin{enumerate}
    \item In \textbf{AF} spaces, we introduce some analogous versions of the Kadets pseudometric on the closure of $\mathrm{Age}(E)$ with respect to the topology induced by the Banach-Mazur distance, and we prove that these pseudometrics are in fact metrics when $E$ is weak \textbf{AF}.
    \item We prove that for every non-free ultrafilter $\mathcal U$ on $\mathbb N$, the ultrapower $E_\mathcal U$ of an \textbf{AF} space $E$ is \textbf{AF} and  \textbf{UH}.
\end{enumerate}
 Answering one of our questions relative to the properties of Fra\"iss\'e spaces listed above, item (2)  indicates that a formally weaker property than Fra\"iss\'e-ness is sufficient to obtain ultrahomogeneous ultrapowers. Also, an interesting aspect of these weak Fra\"iss\'e property is that it will be characterized by properties of 
the compact spaces $\overline{{\rm Age}_k(E)}^{{\rm BM}}$, instead of the possibly non closed ${\rm Age}_k(E)$.

We end the paper by giving a proof of a characterization of Fra\"iss\'e spaces through $\omega$-categoricity of Banach spaces. To fix ideas, let $n\in\mathbb N$ and $E$ be a Banach space. Consider the natural action of ${\rm Isom}(E)$ on $S_E^n\coloneqq S_E\times\stackrel{n}{\cdots}\times S_E$, that is,
\begin{gather}\label{action on S_E}
    (T,(x_1,\ldots,x_n))\mapsto(Tx_1,\ldots,Tx_n).
\end{gather}
Take a norm on $E^n$ such that the above action is isometric (for example the $\ell_2$-sum norm). A Banach space $E$ is called \textit{$\omega$-categorical} if for each $n\in\mathbb N$, the metric quotient of $S_E^n/\mathrm{Isom}(E)$ is compact; see the very general study of I. Ben-Yaacov et al in \cite{ben-yaacov I}. 
All $L_p(0,1)$-spaces, $1 \leq p<+\infty$, and the Gurarij space $\mathbb G$ are examples of $\omega$-categorical \cite[Section 17]{ben-yaacov I} and \cite[Section 2]{ben-yaacov II}, respectively. 

In the last section, we establish the following result, which was claimed but not explicitely proved by I. Ben Yaacov \cite{ben-yaacov III}: a Banach space is Fra\"iss\'e if and only if it is $\omega$-categorical and approximately ultrahomogeneous. As observed by him, this gives a model theoretic proof of the result from \cite{ferenczi et all}
that the spaces $L_p(0,1)$ are Fra\"iss\'e when $p=2$ or $p\not\in 2\mathbb N$. 
 
\section{Finitely isometrically extensible Banach spaces}

A drawback of the ultrahomogenety properties such as $\textbf{UH}$, $\textbf{AUH}$, or the Fra\"iss\'e property, is that they are not preserved under norm $1$ projections. More regularity might be desirable for abstract results about classes of spaces. For this reason we consider the following:

 \begin{definition}
A Banach space $E$ is \textit{finitely isometrically extensible} (\textbf{FIE}) if for all $X\in\mathrm{Age}(E)$ and all $\phi\in\mathrm{Emb}(X,E)$, there is an one-norm operator from $E$ to $E$ which extends $\phi$.
\end{definition}

Note that the \textbf{FIE}-property is equivalent to the following: 
any isometry between finite dimensional subspaces of $E$ can be extended to a one-norm operator from $E$ to $E$. 

\begin{example}
Clearly, \textbf{UH} Banach spaces enjoy the \textbf{FIE}-property. In particular Hilbert spaces are \textbf{FIE}. 
\end{example}

\begin{example}
Any 1-universally separably injective is \textbf{FIE}. For instance
$C(K)$, where $K$ is an extremely disconnected compact Hausdorff space, is \textbf{FIE} since every Banach space in this class is 1-universally separably injective \cite[Proposition 1.19]{aviles et al}. 
\end{example}

The previous example is a particular case of the following general statement. Recall that for a non-empty set $\Gamma$, $c_0(\Gamma,X)$ denotes the Banach space of all maps $f\colon\Gamma\to X$ with the property that for each $\varepsilon>0$, the set $\{\gamma\in\Gamma\,\colon\, \|f(\gamma)\|\geq\varepsilon\}$ is finite, endowed with the supremum norm. When $X$ is the scalar field, we just write $c_0(\Gamma)$.

\begin{example}\label{FIE-ness of sums of c_0(a)}
If $\Lambda$ and $\Gamma$ are non-empty sets, then $\ell_\infty(\Lambda,c_0(\Gamma,X))$ is \textbf{FIE} whenever $X$ is 1-universally separably injective.

Indeed, let $E\in\mathrm{Age}(\ell_\infty(\Lambda,c_0(\Gamma,X)))$ and $T\in\mathrm{Emb}(E, \ell_\infty(\Lambda,c_0(\Gamma,X)))$ be given. Note that $T$ determines a family of operators
$\{T^\lambda\}_{\lambda\in\Lambda}$  from $E$ to $c_0(\Gamma,X)$ such that $Tx=(T^\lambda(x))_{\lambda\in\Lambda}$ for all $x\in E$ and 
$\|T\|=\sup_{\lambda\in\Lambda}\|T^\lambda\|$.

If $\lambda\in\Lambda$ is given, there is a family $(T_\gamma^\lambda)_{\gamma\in\Gamma}$ of operators from $E$ to $X$ such that $T^\lambda(x)=(T_\gamma^\lambda(x))_{\gamma\in\Gamma}$ for all $x\in E$.  Also note that $T_\gamma^\lambda\to{\bf0}$ in the SOT-topology and since $E$ is finite-dimensional, $T_\gamma^\lambda\to{\bf0}$ in the norm operator topology.
For each $\gamma\in\Gamma$, let $\tilde T_\gamma^\lambda\colon \ell_\infty(\Lambda,c_0(\Gamma,X))\to X$ be a norm-preserving extension of $T_\gamma^\lambda$ and define $\tilde T^\lambda\colon \ell_\infty(\Lambda,c_0(\Gamma,X))\to c_0(\Gamma,X)$ by $\tilde T^\lambda(x)=(\tilde T_\gamma^\lambda(x))_{\gamma\in\Gamma}$ if $x\in \ell_\infty(\Lambda,c_0(\Gamma,X))$. Clearly, $\tilde T^\lambda$ is linear and $\|\tilde T^\lambda\|=\sup_{\gamma\in\Gamma}\|\tilde T_\gamma^\lambda\|\leq1$. Finally, let $\tilde T\colon \ell_\infty(\Lambda,c_0(\Gamma,X))\to \ell_\infty(\Lambda,c_0(\Gamma,X))$ be given by $\tilde T(x)=(\tilde T^\lambda(x))_{\lambda\in\Lambda}$ 
for $x\in \ell_\infty(\Lambda,c_0(\Gamma,X))$. It is easy to see that $\tilde T$ extends $T$ and $\|\tilde T\|=1$.
\end{example}

\begin{proposition}\label{fie-ness y complementos}
Let $E$ be a \textbf{FIE} Banach space. If $F$ is a 1-complemented subspace of $E$, then $F$ is \textbf{FIE}.
\end{proposition}

\begin{proof}
    Let $W\in\mathrm{Age}(F)$ and $T\in\mathrm{Emb}(W,F)$ be given. By the \textbf{FIE}-ness of $E$, there is a
    one-norm operator $\hat T\colon E\to E$ which extends $T$. Let $P_F\colon E\to F$ be a 1-projection. We set $\tilde T=P_F\circ\hat T\circ i_F\colon F\to F$, where $i_F\colon F\to E$ is the natural inclusion. Clearly, $\tilde T$ extends $T$ and $\|\tilde T\|=1$.
\end{proof}

\section{Almost ultrahomogeneous and almost finitely isometrically extensible Banach spaces}

In order to state the next results, we introduce the weakening of ultrahomogeneity that is one of our main focus in this paper, and is closer to the original definition of the Gurarij space as well to the way the results of Lusky \cite{lusky} regarding $L_p(0,1)$-spaces were originally formulated.
\begin{definition}
 A Banach space is called \textit{almost ultrahomogeneous} (almost \textbf{UH}), if for all $\varepsilon>0$ and all $X\in\mathrm{Age}(E)$ and all $\phi\in\mathrm{Emb}(X,E)$, there exists $T\in\mathrm{Isom}_\varepsilon
(E)$ such that $T|_E=\phi$.
\end{definition}

\begin{remark}\label{AUH implies UH}
Every \textbf{AUH} Banach space is almost \textbf{UH}. Indeed, assume that $E$ is \textbf{AUH} and let 
$\varepsilon>0$, $X\in\mathrm{Age}(E)$ and $\phi\in\mathrm{Emb}(X,E)$ be given. Since $E$ is \textbf{AUH}, there is $U\in\mathrm{Isom}(E)$ such that $\|U|_X-\phi\|<\varepsilon/(2\dim X)$. If $P_X\colon E\to X$ is a projection with $\|P_X\|\leq\dim X$, we set $T=U-(U-\phi)\circ P_X.$ It is not difficult to check that $T\in\mathrm{Isom}_\varepsilon(E)$ and $T|_X=\phi$.
\end{remark}

To relate the 
almost \textbf{UH} property to the \textbf{FIE}-property, we need the following
approximate version of this property.

\begin{definition}
A Banach space $E$ has the \textit{almost finitely isometrically extensible} (\textbf{aFIE}) if it satisfies the following condition: for all $\varepsilon>0$ and all $X\in\mathrm{Age}(E)$ and all $\phi\in\mathrm{Emb}(X,E)$, there exists $T\colon E\to E$ such that $T|_E=\phi$ and $\|T\|\leq1+\varepsilon$.
\end{definition}

Note that every \textbf{FIE} Banach space is \textbf{aFIE}. 
This notion easily relates to ultrahomogeneity properties as follows:

\begin{fact} Any almost \textbf{UH} space is \textbf{aFIE}. \end{fact} 

\begin{example}
For $p\in 2\mathbb N$ and $p\geq4$, $L_p(0,1)$ is not \textbf{aFIE}. Actually, for any $C>1$, by \cite[Proposition 2.10]{ferenczi et all}, there is a finite dimensional subspace such that if $T\colon X\to L_p(0,1)$ is an isometry and if $\tilde T$ extends $T$, then $\|\tilde T\|\geq C$.
\end{example}

The upcoming result is the corresponding version of Proposition \ref{fie-ness y complementos} for \textbf{aFIE} Banach spaces.

\begin{proposition}\label{afie-ness y complementos}
Let $E$ be a Banach space and $F$ be a subspace of $E$.
\begin{enumerate}
    \item If $E$ is \textbf{aFIE} and $F$ is 1-complemented in $E$, then $F$ is \textbf{aFIE}.
    \item If $E$ is \textbf{aFIE} and $F$ is 1-complemented in $E^{**}$, then $F$ is \textbf{FIE}. In particular, reflexive \textbf{aFIE} spaces are \textbf{FIE}.
\end{enumerate}
\end{proposition}

\begin{proof}
\begin{enumerate}
  \item If $\varepsilon>0$, $Y\in\mathrm{Age}(F)$ and $T\in\mathrm{Emb}(Y,F)$ are given, by the \textbf{aFIE}-ness of $E$, there is an operator $\hat T\colon E\to E$ such that $\hat T|_Y=T$ and $\|\hat T\|\leq1+\varepsilon$. Let $P_F\colon E\to F$ be a 1-projection. We set $\tilde T=P_F\circ\hat T\circ i_F\colon F\to F$, where $i_F\colon F\to E$ is the natural inclusion. Clearly, $\tilde T$ extends $T$ and $\|\tilde T\|\leq1+\varepsilon$.
    \item  Let $Y\in\mathrm{Age}(F)$ and $T\in\mathrm{Emb}(Y,F)$ be given. For each $Z\in\mathrm{Age}(F)$ with $Y\subset Z$, there exists $T_Z\colon Z\to E$ extending $T$ such that $\|T_Z\|\leq1+\frac{1}{\dim Z}$. Let $\mathcal U$ be an ultrafilter on the set of finite dimensional subspaces of $F$ containing $Y$ and refining the Fréchet filter and define $\psi\colon F\to E^{**}$ by $\psi(y)=\displaystyle w^*-\lim_{\mathcal U}T_Z(y)$, $y\in F$. The Banach-Alaoglu theorem ensures that $\psi$ is well-defined. If $P\colon E^{**}\to F$ is a 1-projection, the operator $\tilde T\colon F\to F$ given by $\tilde T=P\circ\psi$ satisfies the requirements.\qedhere
\end{enumerate}
\end{proof}

\begin{remark}\label{L_p is FIE}
 Since $L_p(0,1)$ is \textbf{AUH} (and therefore \textbf{aFIE}) for $p\not\in 2\mathbb N+4$ \cite{lusky}, Proposition \ref{afie-ness y complementos} implies that
 $L_p(0,1)$ is \textbf{FIE} for $p\not\in 2\mathbb N+4$.
 More generally 
 Proposition \ref{afie-ness y complementos}(2) implies that all separable $L_p(\mu)$ are \textbf{FIE} when $p\not\in2\mathbb N+4$.
\end{remark}

Other examples of \textbf{FIE} spaces are the so called \textit{1-uniformly finitely extensible} spaces. If $\lambda\geq1$, a Banach space $E$ is called $\lambda$-uniformly finitely extensible ($\lambda$-\textbf{UFO}) if for all finite dimensional subspace $X$ of $E$, each operator
 $\tau\colon X\to E$ can be extended to an operator $T\colon E\to E$ with $\|T\|\leq\lambda\|\tau\|$. $\lambda$-\textbf{UFO} were introduced by Y. Moreno and A. Plichko in \cite{moreno} and 
  systematically studied in \cite{castillo} and \cite{castillo-moreno-ferenczi}. It is worth mentioning that $\lambda$-\textbf{UFO} spaces satisfy the following dichotomy: every $\lambda$-\textbf{UFO} space is either an $\mathcal L_\infty$-space or a weak type 2 near-Hilbert space with the Maurey projection property \cite[Theorem 5.1]{castillo-moreno-ferenczi}.

\

The following diagram displays the
basic implications between the multidimensional properties considered so far:

\[
\begin{array}
[c]{cccccccc}
  \mbox{Ultrahomogeneous}   \Rightarrow & \mbox{Approximately Ultrahomogeneous} & \Rightarrow  \mbox{Almost Ultrahomogeneous} \\
\Downarrow  & & \Downarrow\\
\mbox{\textbf{FIE}} & \Rightarrow &   \mbox{\textbf{aFIE}}\\
\end{array}
\]

\

While $\textbf{FIE}$ and $\textbf{aFIE}$ are equivalent for reflexive spaces, we do not know whether there are $\textbf{aFIE}$-spaces which are not $\textbf{FIE}$. We also do not know whether the Approximate and  the Almost Ultrahomogeneities are equivalent properties.
All other implications in this diagram are strict.

\section{Almost Fra\"iss\'e Banach spaces}

The property introduced in this section is inspired by the recent notion of Fra\"iss\'eness studied in \cite{ferenczi et all}. We start by recalling the Fra\"iss\'e and the weak Fra\"iss\'e properties for Banach spaces.

\begin{definition}\cite[Definition 2.2]{ferenczi et all}\label{definition Fraisse}
Let $E$ be Banach space. 
\begin{enumerate}
\item $E$ is \textit{weak Fra\"iss\'e} if for every $\varepsilon>0$ and every 
$X\in\mathrm{Age}(E)$
there is $\delta>0$ such that if 
 $\phi,\psi\in\mathrm{Emb}_\delta(X,E)$, then there exists $T\in\mathrm{Isom}(E)$ with $\|T\circ\phi-\psi\|<\varepsilon$.

    \item $E$ is \textit{Fra\"iss\'e} if for every $\varepsilon>0$ and every dimension $k\in\mathbb N$
there is $\delta>0$ such that if $X\in\mathrm{Age}_k(E)$ and 
 $\phi,\psi\in\mathrm{Emb}_\delta(X,E)$, then there exists $T\in\mathrm{Isom}(E)$ with $\|T\circ\phi-\psi\|<\varepsilon$.
\end{enumerate}
\end{definition}

In \cite{ferenczi et all} it was proved that the Gurarij space $\mathbb G$ and $L_p(0,1)$ for $p\not\in2\mathbb N+4$ are Fra\"iss\'e Banach spaces. While the Fra\"iss\'e property obviously implies the weak Fra\"iss\'e property, it is not known whether the two properties coincide.
Recall that the Banach-Mazur distance on $\mathrm{Age}_k(E)$ is defined by
\begin{gather*}
    d_{\mathrm{BM}}(X,Y)=\inf\{\log(\|T\|\|T^{-1}\|)\,\colon\,\mbox{$T\colon X\to Y$ is an isomorphism}\}.
\end{gather*}
The authors of \cite{ferenczi et all} actually prove that a Banach space is Fra\"iss\'e if and only if it is weak Fra\"iss\'e and $(\mathrm{Age}_k(E),d_{\mathrm{BM}})$ is a compact metric space for each $k\in\mathbb N$ \cite[Theorem 2.12]{ferenczi et all}.
Our guideline is now to investigate whether a similar result holds for natural variations of these properties discussed in the introduction and we shall now define.

It will be important to recall that $(\overline{\mathrm{Age}_k(E)}^{\mathrm{BM}},d_{\mathrm{BM}})$ is a compact metric space for each $k\in\mathbb N$.
While the Fra\"iss\'e and weak Fra\"iss\'e properties involved only elements of $\mathrm{Age}(E)$, an interesting fact is that we shall actually consider properties where elements of  $\overline{\mathrm{Age}(E)}^{{\rm BM}}$ can be relevant as well.

\begin{proposition}\label{equivalent definitions of Fraisse}
Let $E$ be a Banach space. The following statements are equivalent:
 \begin{enumerate}
  \item whenever $X \in \overline{\mathrm{Age}(E)}^{\rm{BM}}$ and $\varepsilon>0$, there exists $\delta>0$ such that if $\phi_1,\phi_2\in{\rm Emb}_\delta(X,E)$, there exists $T \in {\rm Isom}_{\varepsilon}(E)$ such that
  $\phi_2=T \circ \phi_1$.
  
\item whenever $k \in\mathbb N$ and  $\varepsilon>0$, there exists $\delta>0$, such that if $X \in \overline{\mathrm{Age}_k(E)}^{\rm{BM}}$ and $\phi_1,\phi_2\in{\rm Emb}_\delta(X,E)$, there exists $T \in {\rm Isom}_{\varepsilon}(E)$ such that
  $\phi_2=T \circ \phi_1$.
\item whenever $k \in\mathbb N$ and  $\varepsilon>0$, there exists $\delta>0$, such that if $X \in \mathrm{Age}_k(E)$ and $\phi_1,\phi_2\in{\rm Emb}_\delta(X,E)$, there exists $T \in {\rm Isom}_{\varepsilon}(E)$ such that
  $\phi_2=T \circ \phi_1$.

\item whenever $k \in\mathbb N$ and $\varepsilon>0$, there exists $\delta>0$, such that for any 
$X \in \mathrm{Age}_k(E)$ and any $\phi \in\mathrm{Emb}_\delta(X,E)$, there exists $T \in {\rm Isom}_\varepsilon(E)$ such that
$T|_{X}=\phi$.
 \end{enumerate}
\end{proposition}

\begin{proof}
The implications (2) $\Rightarrow$ (1), (2) $\Rightarrow$ (3) and (3) $\Rightarrow$ (4) are immediate. We prove (1) $\Rightarrow$ (2) , (4) $\Rightarrow$ (3), and (3) $\Rightarrow$ (2). 

(1) $\Rightarrow$ (2): Suppose that (2) does not hold. So there is $k_0\in\mathbb N$ and $\varepsilon_0>0$ such that for each $n\in\mathbb N$ there exist $X_n\in\overline{\mathrm{Age}_{k_0}(E)}^{{\rm BM}}$ and $\phi_n^1,\phi_n^2\in{\rm{Emb}}_{1/n}(X_n,E)$ satisfying $\phi_n^2\neq T\circ \phi_n^1$ for any $T\in{\rm Isom}_{\varepsilon_0}(E)$. By compactness of $\overline{\mathrm{Age}_{k_0}(E)}^{{\rm BM}}$ we may assume that $X_n\stackrel{{\rm BM}}{\to}X$ for some $X\in\overline{\mathrm{Age}_{k_0}(E)}^{{\rm BM}}$. Now, let $\delta>0$ be the corresponding number satisfying (1) for $X$ and $\varepsilon_1=\varepsilon_0/2$. Choose $0<\xi<\delta$ and let $n_0\in\mathbb N$ be such that $\frac{1}{n_0}<\frac{\delta-\xi}{1+\xi}$ and $d_{\mathrm{BM}}(X_{n_0},X)<\xi$. If $l\colon X\to X_{n_0}$ is an isomorphism with $\|l\|=1$ and $\|l^{-1}\|\leq1+\xi$, then $\phi_{n_0}^1\circ l,\phi_{n_0}^2\circ l\in{\rm{Emb}}_{\delta}(X,E)$. By (1), there exists $T\in\mathrm{Isom}_{\varepsilon_1}(E)$ such that $\phi_{n_0}^2\circ l=T\circ \phi_{n_0}^1\circ l$. Hence $\phi_{n_0}^2=T\circ \phi_{n_0}^1$ for some $T\in{\rm Isom}_{\varepsilon_0}(E)$ which is impossible.

 (4) $\Rightarrow$ (3): Suppose that (4) is valid and let $\varepsilon>0$ and $k\in\mathbb N$ be given. Take $\delta>0$ corresponding to (4) and fix $0<\delta'<\delta$ such that $(1+\delta')^2\leq1+\delta$. Now let $\phi,\psi\in\mathrm{Emb}_{\delta'}(X,E)$ be given. Write $Y=\phi(X)\in\mathrm{Age}_k(E)$ and define $\eta\colon Y\to E$ by $\eta y=\psi(\phi^{-1}y)$, $y\in Y$. Thus, $\eta\in\mathrm{Emb}_{\delta}(Y,E)$ and by (4), there is $T\in\mathrm{Isom}_{\varepsilon}(E)$ such that 
$T|_{Y}=\eta$. The equation $T|_{Y}=\eta$ means that $T\circ\phi=\psi$. 

(3) $\Rightarrow$ (2): Indeed, let $\varepsilon>0$ and $k\in\mathbb N$ be given, fix $\delta>0$ as in (3) and take $0<\xi<\frac{\delta}{2+\delta}$. If $X\in\overline{\mathrm{Age}_k(E)}^{\mathrm{BM}}$, there exists an $\xi$-isometry $A_\xi\colon X\to X_\xi$, where $X_\xi\in\mathrm{Age}_k(E)$. Now, for $\phi,\psi\in\mathrm{Emb}_{\delta/2}(X,E)$,
we have $\phi\circ A_\xi^{-1},\psi\circ A_\xi^{-1}\in\mathrm{Emb}_\delta(X_\xi,E)$. By (3), we have $\psi\circ A_\xi^{-1}=T\circ\phi\circ A_\xi^{-1}$ for some $T\in\mathrm{Isom}_{\varepsilon}(E)$. Thus, $\psi=T\circ\phi$.
\end{proof}

\begin{proposition}\label{equivalent definitions of weak Fraisse}
Let $E$ be a Banach space. The following statements are equivalent:
 \begin{itemize}
  \item[(5)]
  whenever $X \in \mathrm{Age}(E)$ and $\varepsilon>0$, there exists $\delta>0$ such that if $\phi_1,\phi_2\in{\rm Emb}_\delta(X,E)$, there exists $T \in {\rm Isom}_{\varepsilon}(E)$ such that
  $\phi_2=T \circ \phi_1$.
  \item[(6)]
  whenever $X \in \mathrm{Age}(E)$ and $\varepsilon>0$, there exists $\delta>0$ such that if $\phi\in{\rm Emb}_\delta(X,E)$, there exists $T \in {\rm Isom}_{\varepsilon}(E)$ such that
  $T|_{X}=\phi$.
  \end{itemize}
\end{proposition}

\begin{proof}
Suppose that (6) is valid and let $X \in \mathrm{Age}(E)$ and $\varepsilon>0$ be given. If $\xi>0$ satisfies $(1+\xi)^2\leq1+\varepsilon$, take $\delta>0$ corresponding to $\xi$ in (6). Let $\phi_1,\phi_2\in\mathrm{Emb}_{\delta}(X,E)$ be given. From (6) there are
$T_1,T_2\in\mathrm{Isom}_{\xi}(E)$ such that $T_1|_X=\phi_1$ and $T_2|_X=\phi_2$. So, $T=T_2\circ T_1^{-1}\in\mathrm{Isom}_\varepsilon(E)$ and 
$T\circ\phi_1=\phi_2$.
\end{proof}

Propositions \ref{equivalent definitions of Fraisse} and \ref{equivalent definitions of weak Fraisse} together with Definition \ref{definition Fraisse} motivate the next definitions.

\begin{definition}
Let $E$ be a Banach space. 
\begin{enumerate}
    \item $E$ is \textit{almost Fra\"iss\'e} (\textbf{AF}) if it satisfies one of conditions (and hence all of them) in Proposition \ref{equivalent definitions of Fraisse}.

    \item  $E$ is \textit{weak almost Fra\"iss\'e} (weak \textbf{AF}) if it satisfies one of conditions (and hence all of them) in Proposition \ref{equivalent definitions of weak Fraisse}.
\end{enumerate}
\end{definition}

It is obvious, but worth stating, that the \textbf{AF}-property implies the weak \textbf{AF}-property. Items (2)-(3)-(4) in Proposition \ref{equivalent definitions of Fraisse} indicates that some uniformity with respect to the dimension of $E$ follows from the almost Fra\"iss\'e property, while this does not formally follow from the weak almost Fra\"iss\'e property. However,
as a consequence of (1) in Proposition \ref{equivalent definitions of Fraisse} 
and  (5) in Proposition \ref{equivalent definitions of weak Fraisse}, we obtain:

\begin{corollary}\label{AFandwAF}
If $E$ is a weak \textbf{AF} Banach space and $\mathrm{Age}_k(E)$ is compact for all $k$, then
$E$ is \textbf{AF}.
\end{corollary}

It seems to be open whether the \textbf{AF}-property implies that 
$\mathrm{Age}_k(E)$ is compact for all $k$, or whether the weak \textbf{AF} property implies the \textbf{AF} property in general.
The next result justifies the terminology used here.

\begin{proposition}\label{Fraisse implica AF}
Any  Fra\"iss\'e (resp. weak Fra\"iss\'e) Banach space is \textbf{AF} (resp. weak \textbf{AF}).
\end{proposition}

\begin{proof}
Let $\varepsilon>0$ and $k\in\mathbb N$ be given, and take $0<\xi<1$ such that 
$\frac{1+\xi}{1-\xi}<1+\varepsilon$. Fix $\delta>0$ corresponding
to $\xi/k$ in the definition of Fra\"iss\'e. Also let $X\in\mathrm{Age}_k(E)$ and $\phi\in\mathrm{Emb}_\delta(X,E)$.
By the Fra\"iss\'e-ness, there is $S\in\mathrm{Isom}(E)$ satisfying $\|S|_X-\phi\|<\frac{\xi}{k}$. If $\psi:=S|_X-\phi$ and $P\colon E\to X$ is a projection with $\|P\|\leq k$, then by setting $T=S-\psi\circ P\colon E\to E$ we have $T|_X=\phi$ and $\max\{\|T\|,\|T^{-1}\|\}\leq\frac{1+\xi}{1-\xi}<1+\varepsilon$, so $T\in\mathrm{Isom}_{\varepsilon}(E)$. The proof of the second statement is analogous. 
\end{proof}

It is worth noting the following property of \textbf{AF} spaces. Recall that two spaces are said to be almost isometric if they are $\varepsilon$-isometric for all $\varepsilon>0$.

\begin{fact}\label{AF is preserved by almost isometries}
    If $E$ is almost isometric to $F$ and $E$ is \textbf{AF} (weak \textbf{AF}, respectively), then $F$ is \textbf{AF} (weak \textbf{AF}, respectively).
\end{fact}

\begin{proof}
    Let $\varepsilon>0$ and $k\in\mathbb N$ be given, and let $\delta>0$ be the corresponding number to the \textbf{AF} definition. Take $0<\xi<\delta/(2+\delta)$ and let $j\colon E\to F$ be an $\xi$-isometry.
    If $Y\in\mathrm{Age}_k(F)$ and $\phi_1,\phi_2\in\mathrm{Emb}_{\delta/2}(Y,F)$, then $\phi_1\circ j,\phi_2\circ j\in\mathrm{Emb}_\delta(j^{-1}(Y),E)$. Since $E$ is \textbf{AF}, there is $T\in\mathrm{Isom}_\varepsilon(E)$ such that $T\circ\phi_1\circ j=\phi_2\circ j$, that is, $T\circ\phi_1=\phi_2$. The statement for weak \textbf{AF} spaces is proved similarly.
\end{proof}

On the other hand, we do not know if the class of \textbf{FIE} or \textbf{aFIE} spaces is stable by almost isometries.

We end this part by displaying the relationships between the notions introduced throughout the paper.

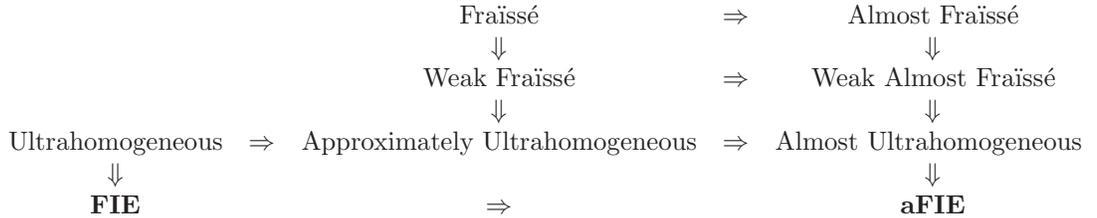
\[
\begin{array}
[c]{cccccccc}
& &  \mbox{Fra\"iss\'e} &  \Rightarrow & \mbox{Almost Fra\"iss\'e} \\
& & \Downarrow  & &  \Downarrow \\
& & \mbox{Weak Fra\"iss\'e}  & \Rightarrow & \mbox{Weak Almost Fra\"iss\'e} \\
& & \Downarrow & & \Downarrow\\
 \mbox{Ultrahomogeneous} & \Rightarrow & \mbox{Approximately Ultrahomogeneous} & \Rightarrow &  \mbox{Almost Ultrahomogeneous } \\
\Downarrow & & & & \Downarrow\\
\mbox{\textbf{FIE}} & & \Rightarrow & & \mbox{\textbf{aFIE}} \\
\end{array}
\]

We know of no example which is almost ultrahomogeneous but not Fra\"iss\'e. So the six properties in the upper right corner could be equivalent; one important aspect of this question relates to whether ${\rm Age}(E)$ is closed with respect to the Banach-Mazur pseudodistance. Under this hypothesis, Fra\"iss\'e and its weak version are equivalent (Theorem 2.12 in \cite{ferenczi et all}), and almost Fra\"iss\'e and its weak version are equivalent (Corollary \ref{AFandwAF}).

\subsection{Some pseudodistances associated to \textbf{AF} Banach spaces}\label{subsections pseudodistances}

To continue, we recall the Gromov-Hausdorff function on $\mathrm{Age}_k(E)^2$ introduced in \cite{ferenczi et all} in order to study Fra\"iss\'e Banach spaces. If $X,Y\in\mathrm{Age}_k(E)$, the authors define 
\begin{gather}\label{g-h function}
    \gamma_E(X,Y)=\inf\{d_H(B_{X_0},B_{Y_0})\,\colon\,\mbox{$X_0\equiv X$ and $Y_0\equiv Y$}\},
\end{gather}
where $d_H$ is the  $\|\cdot\|_E$-Hausdorff metric and the symbol $\equiv$ means ``isometric to". Another function defined on $\mathrm{Age}_k(E)^2$ and considered in \cite{ferenczi et all} is
\begin{gather*}
    D_E(X,Y)=\inf\{d_H(B_{TX},B_Y)\,\colon\,T\in\mathrm{Isom}(E)\},\quad X,Y\in\mathrm{Age}_k(E).
\end{gather*}
It is easy to see that $D_E$ is always a pseudometric on $\mathrm{Age}_k(E)$, that $\gamma_E \leq D_E$, and that both are invariant under the action of $\mathrm{Isom}(E)$. Also, $D_E$ is a complete metric for every Banach space $E$ and $\gamma_E=D_E$ when $E$ is \textbf{AUH},  \cite[Proposition 2.14]{ferenczi et all} - and therefore
$\gamma_E$ is complete when $E$ is \textbf{AUH}, as is implicitly used in \cite{ferenczi et all}. The $D_E$-completeness is consequence of the following fact which is certainly well-known but we include a proof. 

\begin{lemma}\label{pseudometric complete induced by G}
    Let $(X,d)$ be a complete metric space. Suppose that $G$ is a group acting on $X$. If $d$ is $G$-invariant, then 
    \begin{align*}
       \rho\colon X\times X&\to\mathbb R\\
        (x,y)&\mapsto\rho(x,y)=\inf_{g\in G}d(gx,y)
    \end{align*}
    is a complete pseudometric.
\end{lemma}

\begin{proof}
It is not difficult to check that $\rho$ is a pseudometric. Now, we prove the completeness.
    Let $(x_n)$ be a sequence in $X$ such that $\rho(x_n,x_{n+1})<1/2^n$ for each $n\in\mathbb N$. We shall contruct sequences $(z_n)$ in $X$ and $(g_n)$ in $G$ such that $g_nz_n=x_n$ and $d(z_n,z_{n+1})<1/2^n$ for each $n\in\mathbb N$. Suppose that $z_1,\ldots,z_n\in X$ and $g_1,\ldots,g_n\in G$ have already been defined. Since $\rho(x_n,x_{n+1})<1/2^n$, there exists $h_n\in G$ satisfying $d(h_nx_n,x_{n+1})<1/2^n$. By the $G$-invariance of $d$ we have 
 \begin{align*}
d(h_nx_n,x_{n+1})=d(x_n,h_n^{-1}x_{n+1})&=d(g_nz_n,h_n^{-1}x_{n+1})\\
     &=d(z_n,g_n^{-1}h_n^{-1}x_{n+1})<1/2^n.
 \end{align*}
   By setting $z_{n+1}=g_n^{-1}h_n^{-1}x_{n+1}$ and $g_{n+1}=h_ng_n$, we end the induction. 

   If $z\in X$ satisfies $z_n\to z$, then $\rho(x_n,z)\leq d(z_n,z)\to0$ as $n\to\infty$. Therefore, $\rho$ is complete.
\end{proof}

\begin{proposition}\label{age(E) is complete when E is auh}
If $E$ is a Banach space, then 
$(\mathrm{Age}(E),D_E)$ is complete.  If $E$ is \textbf{AUH}, then the function $\gamma_E$ coincides with $D_E$, and is therefore a complete metric. 
\end{proposition}

\begin{proof}
The second part is \cite[Proposition 2.14]{ferenczi et all}. For the first part,
let $K(E)$ be the family of nonempty compact subsets of $E$. Since the map $\Psi\colon X\in\mathrm{Age}(E)\mapsto B_X\in K(E)$ is injective and $(K(E),d_H)$ is complete, it suffices to check that $\Psi(\mathrm{Age}(E))$ is $d_H$-closed to prove that it is $d_H$-complete.

Let $(X_n)$ be a sequence in $\mathrm{Age}(E)$ such that $\Psi(X_n)=B_{X_n}\stackrel{d_H}{\to} A$. In particular there exists $k$ such that $X_n \in \mathrm{Age}_k(E)$ for each $n\in\mathbb N$. Note that the limit $A$ is a non-empty balanced compact convex subset of $E$. So $X_0=\bigcup_{\lambda>0}\lambda A$ is a subspace of $E$ and $\Psi(X_0)=B_{X_0}=A$.
Finally, by taking $n_0\in\mathbb N$ such that $d_H(B_{X_{n_0}},B_{X_0})<1/2k$, we obtain that $\dim X_0=k$, that is, $X_0\in\mathrm{Age}_k(E)$. 

Now, since $d_H$ is $\mathrm{Isom}(E)$-invariant, the conclusion of the Proposition follows from Lemma \ref{pseudometric complete induced by G}.
\end{proof}

Now, following the above ideas, we introduce  ``almost" versions $D_E^a$ of $D_E$, and $\gamma_E^a$ of $\gamma_E$. These functions are   intended to be relevant to the case when $E$ is almost \textbf{UH} (and not necessarily \textbf{UH}), or weak \textbf{AF} (and not necessarily weak Fra\"iss\'e). While $D_E^a$ will be defined on $\mathrm{Age}_k(E)$, an interesting new feature of $\gamma_E^a$ is that it will be defined on $\overline{\mathrm{Age}(E)}^{\mathrm{BM}}$.

\begin{definition}
If $X,Y\in\mathrm{Age}(E)$, we let 
\begin{gather*}
    D_\delta(X,Y)=\inf\{d_H(T(B_{X}),U(B_{Y}))\,\colon\,T,U\in\mathrm{Isom}_\delta(E)\}
\end{gather*}
and 
\begin{gather*}
D_E^a(X,Y)\coloneqq\lim_{\delta\to0}D_\delta(X,Y)=\sup_{\delta>0}D_\delta(X,Y).
\end{gather*}
\end{definition}

We note the following:
\begin{fact}
Let $E$ be a Banach space. Then $D_E^a(X,Y)$ is a pseudometric on
$\mathrm{Age}(E)$ such that $D_E^a(X,Y) \leq D_E(X,Y)$.
\end{fact}

\begin{proof} We use the immediate fact: if $A,B \subset E$ are compact and $T \in \mathrm{Isom}_\delta(E)$, then 
$d_H(T(A),T(B)) \leq (1+\delta)d_H(A,B)$. From this we obtain for each $\varepsilon,\delta>0$
$$D_\delta(X,Y)+D_\delta(Y,Z)+\varepsilon \geq (1+\delta)^{-1} D_{\delta^2+2\delta}
(X,Z).$$ The triangle inequality follows when $\varepsilon,\delta$ tends to zero. 
 \end{proof}

In contrast to $D_E$, we have no reason to think that $D_E^a$ is a complete pseudo metric in general. We now turn to the definition of the function $\gamma_E^a$.

\begin{definition} Let $E$ be a Banach space.
If $X,Y\in\overline{\mathrm{Age}_k(E)}^{\mathrm{BM}}$, we set
\begin{gather*}
         d_\delta(X,Y)=\inf\{d_H(t(B_{X}),t'(B_{Y}))\,\colon\,t\in\mathrm{Emb}_\delta(X,E),t'\in\mathrm{Emb}_\delta(Y,E)\}
\end{gather*}
and 
\begin{gather*}
\gamma_E^a(X,Y)\coloneqq\lim_{\delta\to0}d_\delta(X,Y)=\sup_{\delta>0}d_\delta(X,Y).
\end{gather*}
\end{definition}

\begin{lemma}\label{alternative definitions of gamma_E^a}
The following functions give alternative definitions of $\gamma_E^a$: for $X,Y\in\overline{\mathrm{Age}_k(E)}^{\mathrm{BM}}$, we set
\begin{align*}
     d_\delta^1(X,Y)&=\inf\{d_H(B_{tX},B_{sY})\,\colon\,t\in\mathrm{Emb}_\delta(X,E),s\in\mathrm{Emb}_\delta(Y,E)\},\quad\mbox{and}\\
     d_\delta^2(X,Y)&=\inf\{d_H(t(B_{X}),B_{sY})\,\colon\,t\in\mathrm{Emb}_\delta(X,E),s\in\mathrm{Emb}_\delta(Y,E)\}.
\end{align*}
Then 
\begin{gather*}
\gamma_E^a(X,Y)=\lim_{\delta\to0}d_\delta^1(X,Y)=\lim_{\delta\to0}d_\delta^2(X,Y).
\end{gather*}
Moreover, if $X,Y\in\mathrm{Age}(E)$ and 
\begin{gather*}
     D_\delta^1(X,Y)=\inf\{d_H(B_{UX},B_{VY})\,\colon\,U,V\in\mathrm{Isom}_\delta(E)\},
\end{gather*}
then 
\begin{gather*}
D_E^a(X,Y)=\lim_{\delta\to0}D_\delta^1(X,Y)
=\sup_{\delta>0}D_\delta^1(X,Y).
\end{gather*}
\end{lemma}

\begin{proof}
If $u\colon E\to F$ is a $\delta$-isometry, then $\dfrac{1}{1+\delta}B_{uE}\subset u(B_E)\subset (1+\delta)B_{uE}$. So, $d_H(u(B_E),B_{uE})\leq\delta$. Hence 
if $t\in\mathrm{Emb}_\delta(X,E)$ and $s\in\mathrm{Emb}_\delta(Y,E)$, then
\begin{align*}
      d_H(t(B_X),s(B_Y))&\leq 2\delta+d_H(B_{tX},B_{sY}),\quad\mbox{and}\\
      d_H(B_{tX},B_{sY})&\leq 2\delta+d_H(t(B_X),s(B_Y)).
\end{align*}
Thus $\gamma_E^a(X,Y)=\displaystyle\lim_{\delta\to0}d_\delta^1(X,Y)$. The other statements are proved similarly. 
\end{proof}

Before comparing this function to the more classical ones, we observe an easy  consequence of its definition.
It is inspired from \cite[Proposition 2.14]{ferenczi et all}, where it was proved that $d_{\mathrm{BM}}(X,Y)\leq 4kd_H(B_X,B_Y)$ for each $X,Y\in\mathrm{Age}_k(E)$ such that $d_H(B_X,B_Y)<1/2k$.

\begin{lemma}\label{continuity of BM respect d_E}
Let $E$ be a Banach space and $X,Y\in\overline{\mathrm{Age}_k(E)}^{\mathrm{BM}}$. 
\begin{enumerate}
\item If $X,Y\in\mathrm{Age}_k(E)$ and $d_H(X,Y)\leq d<1/2k$, then there exists an isomorphism $\lambda\colon X\to Y$ such that $\|\lambda\|\leq1+kd$, $\|\lambda^{-1}\|\leq1/(1-kd)$ and $\|\lambda-\mathrm{Id}\|\leq kd$.
\item Suppose that $X,Y\in\mathrm{Age}_k(E)$ and let $\sigma>0$ be given. If $\lambda\colon X\to Y$ is an $\varepsilon$-perturbation of ${\rm Id}$ with $k\varepsilon\leq\sigma/(2+\sigma)$, there exists $T\in\mathrm{Isom}_{\sigma}(E)$ such that $T|_X=\lambda$.

    \item 
$d_{\mathrm{BM}}(X,Y)\leq 4k \gamma_{E}^a(X,Y)$ whenever $\gamma_E^a(X,Y)<1/2k$, and 
 $\gamma_E^a(X,Y)=0$ if and only if $X \equiv Y$.
 \item $\gamma_E^a$ is $\mathrm{BM}$-lower semicontinuous on $\overline{\mathrm{Age}_k(E)}^{\mathrm{BM}}$, in the sense
 that if \linebreak$\lim_n d_{\mathrm{BM}}(X_n,X)=0$ and $\lim_n d_{\mathrm{BM}}(Y_n,Y)=0$, then
 \begin{gather*}
     \gamma_E^a(X,Y) \leq \liminf \gamma_E^a(X_n,Y_n).
 \end{gather*}
\end{enumerate}
\end{lemma}

\begin{proof} 
\begin{enumerate}
\item Fix an Auerbach basis of $X$, $\{x_1,\ldots,x_k\}$. From definition of $d_H$, for each $j=1,\ldots,k$ there are $y_j\in Y$ satisfying $\|x_j-y_j\|\leq d$. If $a_1,\ldots,a_k\in\mathbb K$, we have
   \begin{gather*}
       (1-kd)\left\|\sum_{j=1}^ka_jx_j'\right\|\leq\left\|\sum_{j=1}^ka_jy_j'\right\|\leq(1+kd)\left\|\sum_{j=1}^ka_jx_j'\right\|.
   \end{gather*}
   Thus, $\lambda\colon X\to Y$ defined linearly by $x_j\in X\mapsto y_j\in Y$ is an isomorphism with  $\|\lambda\|\leq1+kd$, $\|\lambda^{-1}\|\leq1/(1-kd)$
   and $\|\lambda-\mathrm{Id}\|\leq kd$.

\item If $P_X\colon E\to X$ is a projection with $\|P_X\|\leq k$, let $T\colon E\to E$ be defined by $T={\rm Id}-({\rm Id}_X-\lambda)\circ P_X$. 
Note that  $T|_X=\lambda$, $\|T\|\leq 1+k\varepsilon\leq 1+\sigma$ and $\|T^{-1}\|\leq 1/(1-k\varepsilon)\leq 1+\sigma$. Hence, $T\in\mathrm{Isom}_\sigma(E)$.

    \item Let $d>0$ be such that $\gamma_E^a(X,Y)<d<1/2k$. Also let $\delta>0$ be such that 
    $d_\delta^1(X,Y)<d<1/2k$. So there are $t\in\mathrm{Emb}_\delta(X,E)$ and $s\in\mathrm{Emb}_\delta(Y,E)$
    satisfying $d_H(B_{tX},B_{sY})<d$. 
    
    Write $X'=tX$ and $Y'=sY$. By Item (1), there exists an isomorphism $\lambda\colon X'\to Y'$ such that 
    $\|\lambda\|\leq1+kd$ and $\|\lambda^{-1}\|\leq1/(1-kd)$.
  Hence,
   $d_{\mathrm{BM}}(X,Y)\leq4\log(1+\delta)+\log(\frac{1+kd}{1-kd})\leq4\log(1+\delta)+ 4kd$. Since $\delta,d$ were arbitrary, we conclude that $d_{\mathrm{BM}}(X,Y)\leq 4k\gamma_E^a(X,Y)$.

    \item Note that if $X,Y, X',Y' \in \overline{\mathrm{Age}_k(E)}^{\mathrm{BM}}$, and $s\colon X \rightarrow X'$,
$t\colon Y \rightarrow Y'$ are $1+\varepsilon$-isometric maps,
then
\begin{align*}
    d_\delta(X',Y')&=\inf\{d_H(B_{uX'},B_{vY'})\,\colon\,u\in\mathrm{Emb}_\delta(X,E),v\in\mathrm{Emb}_\delta(Y,E)\}\\
    &=\inf\{d_H(B_{usX},B_{vtY})\,\colon\,u\in\mathrm{Emb}_\delta(X,E),v\in\mathrm{Emb}_\delta(Y,E)\}\\
    &\geq d_{\delta+\varepsilon+\delta\varepsilon}(X,Y).
\end{align*}
In particular, if $\lim_n d_{\mathrm{BM}}(X_n,X)=0$
and $\lim_n d_{\mathrm{BM}}(Y_n,Y)=0$, then
$$d_{\delta+\varepsilon+\delta\varepsilon}(X,Y) \leq \liminf d_\delta(X_n,Y_n) \leq \liminf \gamma_E^a(X_n,Y_n)$$
and since $\delta$ and $\varepsilon$ were arbitrary,
\begin{gather*}
    \gamma_E^a(X,Y)\leq \liminf \gamma_E^a(X_n,Y_n).\qedhere
\end{gather*}
\end{enumerate}
\end{proof}

Now we list some relationships between the above defined functions and the different forms of ultrahomogeneity.

\begin{proposition}\label{relationships}
 Let $E$ be a Banach space and
     $X,Y\in\mathrm{Age}(E)$. Then 
    \begin{enumerate}
        \item $\gamma_E^a(X,Y) \leq \gamma_E(X,Y) \leq D_E(X,Y)$ and $\gamma_E^a(X,Y) \leq D_E^a(X,Y) \leq D_E(X,Y)$.
        \item If $E$ is almost \textbf{UH}, then $\gamma_E^a(X,Y) \leq D_E^a(X,Y)\leq\gamma_E(X,Y) \leq D_E(X,Y)$.
        \item If $E$ is weak \textbf{AF}, then $\gamma_E^a(X,Y) = D_E^a(X,Y)\leq\gamma_E(X,Y) \leq D_E(X,Y)$, and particular $\gamma_E^a$ is a pseudometric on ${\rm Age}(E)$.
        \item If $E$ is \textbf{AUH}, then $\gamma_E^a(X,Y)  \leq D_E^a(X,Y)\leq \gamma_E(X,Y) = D_E(X,Y)$, and in particular $\gamma_E$ is a pseudometric on ${\rm Age}(E)$.
        \item If $E$ is weak Fra\"iss\'e, then the four maps $\gamma_E^a,\gamma_E, D_E^a, D_E$ coincide on $\mathrm{Age}(E)^2$.
    \end{enumerate}
    \end{proposition}

 \begin{proof} (1) is obvious and (4) was observed in \cite{ferenczi et all}. 
 
 (2) if $t,t'$ are isometric embeddings of $X$ and $Y$ into $E$, and $\delta>0$ is given, let $T,T'\in{\rm Isom}_\delta(E)$ be extensions of $t$ and $t'$ respectively. Then
 $D_\delta(X,Y) \leq d_H(TB_X,T'B_Y)=d_H(tB_X,t'B_Y)$.
 Taking the supremum over $\delta$ and the infimum over $t,t'$ gives that $D_E^a(X,Y) \leq \gamma_E(X,Y)$.
 
(3) If $\varepsilon>0$ is given, let $\delta>0$ be the corresponding number of definition of weak \textbf{AF} to $X$ and $Y$. From its definition there are  $t\in\mathrm{Emb}_\delta(X,E)$ and $s\in\mathrm{Emb}_\delta(Y,E)$ such that 
$d_H(t(B_X),s(B_Y))\leq \gamma_E^a(X,Y)$. Since $E$ is weak \textbf{AF}, there are $T,S\in\mathrm{Isom}_\varepsilon(E)$ which extend $t$ and $s$, respectively. Thus $D_\varepsilon(X,Y)\leq \gamma_E^a(X,Y)$
and the arbitrariness of $\varepsilon>0$ yields $D_E^a(X,Y)\leq \gamma_E^a(X,Y)$. 

(5) Because of (2), it is enough to prove that $D_E \leq \gamma_E^a$.
Let $\varepsilon>0$ and $\delta>0$ be associated number by the weak Fra\"iss\'e property in $X$ and $Y$. If $t\in\mathrm{Emb}_\delta(X,E)$ and $s\in\mathrm{Emb}_\delta(X,E)$ are given, let $T,S \in {\rm Isom}(E)$ be such that $\|T|_X-t\| \leq \varepsilon$ and $\|S|_Y-s\| \leq \varepsilon$. Then
\begin{align*}
        d_H(T(B_X),S(B_Y))
        &\leq d_H(T(B_X),t(B_X))+d_H(S(B_Y),s(B_Y))+d_H(t(B_X),s(B_Y))\\
        &\leq 2\varepsilon+d_H(t(B_X),s(B_Y)).
\end{align*}
Thus, $D_E(X,Y)\leq2\varepsilon+d_\delta(X,Y)\leq2\varepsilon+\gamma_E^a(X,Y)$ . By taking $\varepsilon\to0^+$ we get the result. 
 \end{proof}
 
 We now turn to characterizations of almost Fra\"iss\'e spaces among weak almost Fra\"iss\'e spaces. As we shall see, the situation is more involved than for Fra\"iss\'e spaces, which by \cite[Theorem 2.12]{ferenczi et all} are exactly the weak Fra\"iss\'e spaces for which $\mathrm{Age}_k(E)$ is $\mathrm{BM}$-compact for all $k$.
 
 We consider another natural pseudometric, which is only defined on ${\rm Age}(E)$. 
 
 \begin{definition}
 Let $E$ be a Banach space and $X,Y \in {\rm Age}(E)$. We let 
\begin{gather*}
    D_{{\rm BM}}^E(X,Y)= \inf \{\log(\|T\| \|T^{-1}\|)\,\colon\, \mbox{$T\colon E\to E$ is an isomorphism and $T(X)=Y$}\}.
\end{gather*}
 \end{definition}

 We have the following fact:
  
 \begin{lemma} Let $E$ be a Banach space.
 \begin{enumerate}
 \item $D_{{\rm BM}}^E$ is pseudometric on ${\rm Age(E)}$ dominating $d_{{\rm BM}}$.
 \item There is a constant $c(\delta,k)>0$ such that for any $X,Y \in {\rm Age}_k(E)$, we have
 $D_E^a(X,Y) \leq \delta <1/2k^2\Rightarrow D_{{\rm BM}}^E(X,Y) \leq c(\delta,k)$.
 \item For any $X,Y \in {\rm Age}(E)$, 
 $D_E^a(X,Y)=0 \Leftrightarrow D_{{\rm BM}}^E(X,Y)=0$.
 \item If $E$ is weak \textbf{AF}, then
 ${\rm Id}$ is an homeomorphism between
 $({\rm Age}(E),d_{{\rm BM}})$ and $({\rm Age}(E),D_{{\rm BM}}^E)$.
 \item If $E$ is \textbf{AF}, then
 ${\rm Id}$ is a uniform homeomorphism between
 $({\rm Age}(E),d_{{\rm BM}})$ and $({\rm Age}(E),D_{{\rm BM}}^E)$.
 \end{enumerate}
 \end{lemma}
 
 \begin{proof} (1) follows from definition.
(2) By Lemma \ref{alternative definitions of gamma_E^a} there are $U,V\in\mathrm{Isom}_\delta(E)$ satisfying $d_H(B_{UX},B_{VY})\leq\delta$. If $X'=UX$ and $Y'=VY$, then by Lemma \ref{continuity of BM respect d_E}(1) there is $\lambda\colon X'\to Y'$ which is a $k\delta$-perturbation of $\mathrm{Id}$. Once again by Lemma \ref{continuity of BM respect d_E}(2), there exists $\tilde\lambda\in\mathrm{Isom}_{\sigma(k,\delta)}(E)$ such that
$\tilde\lambda|_X=\lambda$, where $\sigma(k,\delta)=2k^2\delta/(1-k^2\delta)$. If $T=V^{-1}\tilde \lambda U$, then $T\colon E\to E$ is an isomorphism with $TX=Y$ and 
\begin{gather*}
D_{\mathrm{BM}}^E(X,Y)\leq\log(\|T\|\|T^{-1}\|)\leq2\log(1+\delta)+\log\left(\dfrac{1+k^2\delta}{1-k^2\delta}\right):=c(\delta,k).
\end{gather*}
 (3) follows from (2) since $\displaystyle\lim_{\delta\to0}c(\delta,k)=0$. (4) and (5) are obvious implications of the definition.
 \end{proof}

 \begin{question}
     When does $D_{\rm BM}^E$ be a complete pseudometric?
 \end{question}
 
Note that if $D_{{\rm BM}}^E$ is complete and $E$ is \textbf{AF}, then ${\rm Age}(E)$ is ${\rm BM}$-compact, and probably we also have that
 \textbf{AF} is equivalent to weak \textbf{AF} plus $\mathrm{Age}(E)$ ${\rm BM}$-compact.
 
 Recall that  for $X,Y\in\overline{\mathrm{Age}_k(E)}^{\mathrm{BM}}$ we have $\gamma_E^a(X,Y)=0$ if and only if $X \equiv Y$ (Lemma \ref{continuity of BM respect d_E}), and
 if $E$ is  weak \textbf{AF}, then we have  that  $\gamma_E^a=\sup_{\delta>0}d_\delta$ is a pseudometric on $\mathrm{Age}(E)$ coinciding with
 $D_E^a=\sup_{\delta>0}D_\delta$ (Item 3 in Proposition \ref{relationships}). 
 
Note that although $\gamma_E^a$ is defined on $\overline{\mathrm{Age}_k(E)}^{\mathrm{BM}}$, it is not clear whether it is a pseudometric there. We start with a lemma which implies that when $E$ is \textbf{AF},
then $\gamma_E^a$ is indeed a complete pseudometric on 
$\overline{\mathrm{Age}_k(E)}^{\mathrm{BM}}$.

We denote by $(C_k(E),\overline{\gamma_E^a})$ the $\gamma_E^a$-completion of ${\rm Age}_k(E)$, and by $j$ the map $ C_k(E) \rightarrow \overline{{\rm Age}_k(E)}^{\mathrm{BM}}$, defined by
     $j(\overline{(X_n)_n})=\mathrm{BM}-\lim_n X_n$.
Note that $j$ is well defined by uniform continuity of the map ${\rm Id}$ on ${\rm Age}(E)$ with respect to the
$\gamma_E^a$ and $\mathrm{BM}$ pseudometrics.

 \begin{lemma}\label{impli} Let $E$ be a Banach space. Then we have the relation, for $X,Y \in C_k(E)$,
 $$\gamma_E^a(jX,jY) \leq \overline{\gamma_E^a}(X,Y).$$
 Furthermore, consider the properties:
 \begin{enumerate}
     \item $E$ is \textbf{AF}.
     \item $\gamma_E^a$ is a pseudometric on $\overline{\mathrm{Age}_k(E)}^{\mathrm{BM}}$.
     \item 
     The  map 
     $j: C_k(E) \rightarrow \overline{{\rm Age}_k(E)}^{\mathrm{BM}}$ is surjective and satisfies $\gamma_E^a(jX,jY) = \overline{\gamma_E^a}(X,Y)$
     \item $\mathrm{Id}: (\overline{\mathrm{Age}_k(E)}^{\mathrm{BM}},\gamma_E^a) \rightarrow (\overline{\mathrm{Age}_k(E)}^{\mathrm{BM}},\mathrm{BM})$ is uniformly continuous for each $k$.
     
 \end{enumerate}
 Then $(1) \Rightarrow (2) \Leftrightarrow (3) \Rightarrow (4)$.
 
 In particular,  if $E$ is \textbf{AF}, then $\gamma_E^a$ is a complete pseudometric on 
$\overline{\mathrm{Age}_k(E)}^{\mathrm{BM}}$.
 \end{lemma}

 \begin{proof} 
 Write $X=(X_n)_n$, $Y=(Y_n)_n$, where $X_n, Y_n$ are $\gamma_E^a$-Cauchy sequences. Then $X_n$ and $Y_n$ tend to $jX$ and $jY$ respectively with respect to $d_{{\rm BM}}$. By Lemma \ref{continuity of BM respect d_E}(3), it follows that
 $\gamma_E^a(jX,jY) \leq \liminf \gamma_E^a(X_n,Y_n)=\lim_n \gamma_E^a(X_n,Y_n)=\overline{\gamma_E^a}(X,Y)$.
 
 Also, if $(2)$ holds, i.e. $\gamma_E^a$ is a complete pseudometric on 
$\overline{\mathrm{Age}_k(E)}^{\mathrm{BM}}$, and admitting for now $(2) \Rightarrow (3)$, then
  $j$ is a surjective isometry between 
 $(C_k(E),\overline{\gamma_E^a})$ and 
 $(\overline{{\rm Age}_k(E)}^{\mathrm{BM}},\gamma_E^a)$, so 
 $\gamma_E^a$ is necessarily a complete pseudometric on 
$\overline{\mathrm{Age}_k(E)}^{\mathrm{BM}}$, proving the last affirmation of the lemma.

 $(2) \Rightarrow (4)$ is an immediate consequence of Lemma \ref{continuity of BM respect d_E}(3). $(3) \Rightarrow (2)$ is also clear, since if (3) holds then the $\gamma_E^a$-completion $(C_k(E),\overline{\gamma_E^a})$ of ${\rm Age}_k(E)$ coincides with $(\overline{{\rm Age}_k(E)}^{\mathrm{BM}},\gamma_E^a)$, through the map $j$.

$(1) \Rightarrow (2)$: 
 We prove the triangular inequality.
Let $X,Y,Z\in\overline{\mathrm{Age}_k(E)}^{\mathrm{BM}}$ and $r>0$ be fixed. If $0<\varepsilon<r$ is given, let $\delta>0$ be the corresponding value in the definition of \textbf{AF}. 
Also, fix $\delta'>0$ such that $(1+\delta')(1+\varepsilon)<1+r$ and $0<\delta'<\delta$. Then
\begin{gather*}
    d_{\delta'}(X,Y)\leq \gamma_E^a(X,Y)\quad\mbox{and}\quad d_{\delta'}(Y,Z)\leq \gamma_E^a(Y,Z).
\end{gather*}
From definition there are $u\in\mathrm{Emb}_{\delta'}(X,E)$, $v\in\mathrm{Emb}_{\delta'}(Y,E)$, $t\in\mathrm{Emb}_{\delta'}(Y,E)$ and $s\in\mathrm{Emb}_{\delta'}(Z,E)$ such that 
\begin{gather}\label{inequa 1}
    d_H(u(B_{X}),v(B_{Y}))\leq \gamma_E^a(X,Y)\quad\mbox{and}\quad d_H(t(B_{Y}),s(B_{Z}))\leq \gamma_E^a(Y,Z).
\end{gather}
Since $E$ is \textbf{AF}, by (2) of Theorem \ref{equivalent definitions of Fraisse} there is $T\in\mathrm{Isom}_\varepsilon(E)$ with $v=T\circ t$. So,
\begin{gather}\label{inequa 2}
    d_H(Tt(B_{Y}),Ts(B_{Z}))\leq(1+\varepsilon)d_H(t(B_{Y}),s(B_{Z})).
\end{gather}
By combining \eqref{inequa 1} and \eqref{inequa 2} we obtain
\begin{gather*}
    \frac{1}{1+\varepsilon}d_H(v(B_{Y}),Ts(B_{Z}))=\frac{1}{1+\varepsilon}d_H(Tt(B_{Y}),Ts(B_{Z}))\leq \gamma_E^a(Y,Z).
\end{gather*}
By adding the previous inequality and \eqref{inequa 1} it follows that 
\begin{gather*}
   \frac{1}{1+\varepsilon}d_H(u(B_{X}),Ts(B_{Z}))\leq\frac{1}{1+\varepsilon}\gamma_E^a(X,Y)+\gamma_E^a(Y,Z).
\end{gather*}
Since $T\circ s\in\mathrm{Emb}_{r}(Z,E)$ and $u\in\mathrm{Emb}_{\delta'}(X,E)\subset\mathrm{Emb}_{r}(X,E)$, we have 
\begin{gather*}
    \frac{1}{1+\varepsilon}d_r(X,Z)\leq\frac{1}{1+\varepsilon}\gamma^a_E(X,Y)+\gamma_E^a(Y,Z).
\end{gather*}
Since $r,\varepsilon>0$ were arbitrary, we obtain $\gamma_E^a(X,Z)\leq \gamma_E^a(X,Y)+\gamma_E^a(Y,Z).$

$(2) \Rightarrow (3)$: Let $(X_n)_n$ and $(Y_n)_n$ be $\gamma_E^a$-Cauchy sequences in ${\rm Age}(E)$ and $X=j(\overline{(X_n)_n})$,
$Y=j(\overline{(Y_n)_n})$. By using the triangular inequality and
 the lower semicontinuity of $\gamma_E^a$ (Lemma \ref{continuity of BM respect d_E}) we have
\begin{align*}
    |\gamma_E^a(X,Y)-\gamma_E^a(X_n,Y_n)| 
   &\leq \gamma_E^a(X_n,X) + \gamma_E^a(X_n,X) \\
&\leq \liminf_k \gamma_E^a(X_n,X_k)+\liminf_k \gamma_E^a(X_n,X_k).
\end{align*}
 Since $(X_n)_n$ and $(Y_n)_n$ are $\gamma_E^a$-Cauchy, the last inequality  implies that
$$\gamma_E^a(X,Y)=\lim_n \gamma_E^a(X_n,Y_n)=\overline{\gamma}_E^a(\overline{(X_n)_n},\overline{(Y_n)_n}).$$
From the completeness of $(C_k(E),\overline{\gamma_E^a})$ it follows that  $j(C_k(E))$ is closed and since it is dense (it contains ${\rm Age}_k(E)$), $j$ is surjective.
 \end{proof}
 
 We finally obtain a list of equivalent sufficient conditions regarding weak \textbf{AF} spaces. It is not clear however how they relate to the \textbf{AF} property.

\begin{proposition}\label{pseudometric proof}
Let $E$ be a weak \textbf{AF} Banach space such that 
$\gamma_E^a$ is a pseudometric on $\overline{\mathrm{Age}(E)}^{\mathrm{BM}}$. The following statements are equivalent:
\begin{enumerate}
    \item ${\rm Id}\colon ({\rm Age}_k(E),\gamma_E^a) \rightarrow ({\rm Age}_k(E),\mathrm{BM}) $ is a uniform homeomorphism for each $k$
    \item 
    The map
    $${\rm Id}\colon (\overline{\mathrm{Age}(E)}^{\mathrm{BM}},\gamma_E^a) \rightarrow (\overline{\mathrm{Age}(E)}^{\mathrm{BM}},\mathrm{BM})$$ is an homeomorphism.
    \item 
    The map
    $${\rm Id}\colon (\overline{\mathrm{Age}_k(E)}^{\mathrm{BM}},\gamma_E^a) \rightarrow (\overline{\mathrm{Age}_k(E)}^{\mathrm{BM}},\mathrm{BM})$$ is a uniform homeomorphism for each $k$.
    \item $\gamma_E^a$ is a compact pseudometric on $\overline{\mathrm{Age}_k(E)}^{\mathrm{BM}}$ for each $k$. 
    \item 
    The set $(\mathrm{Age}_k(E),\gamma_E^a)$ is totally bounded for each $k$.
\end{enumerate}
\end{proposition}

\begin{proof} (4) $\Rightarrow$ (3) : if (4) holds then by Lemma \ref{continuity of BM respect d_E} (2), ${\rm Id}$ is a uniformly continuous bijection between the compact spaces $(\overline{\mathrm{Age}_k(E)}^{\mathrm{BM}},\gamma_E^a)$ and
$(\overline{\mathrm{Age}_k(E)}^{\mathrm{BM}},\mathrm{BM})$ and therefore a uniform homeomorphism.

(3) $\Rightarrow$ (2) is obvious. (2) $\Rightarrow$ (1): if (2) is valid, the map ${\rm Id}$ is a homeomorphism between compact spaces, and therefore a uniform homeomorphism, implying (1).

(1) $\Rightarrow$ (5): By (1), the completions of ${\rm Age}_k(E)$ with respect to $\mathrm{BM}$ and $\gamma_E^a$ coincide. In particular the $\gamma_E^a$-completion of $({\rm Age}_k(E),\gamma_E^a)$ is compact, and so ${\rm Age}_k(E)$ is totally bounded for $\gamma_E^a$.

(5) $\Rightarrow$ (4): since $\gamma_E^a$ is a pseudometric on $\overline{\mathrm{Age}_k(E)}^{\mathrm{BM}}$, it follows from (2)$\Leftrightarrow$(3) in  Lemma \ref{impli} that the map
$j: C_k(E) \rightarrow 
(\overline{{\rm Age}_k(E)}^{\mathrm{BM}},\gamma_E^a)$ given by $\overline{(X_n)_n} \mapsto \mathrm{BM}-\lim_n X_n$ is a surjective isometry.   Therefore $(\overline{{\rm Age}_k(E)}^{\mathrm{BM}},\gamma_E^a)$ is compact.
\end{proof}

\begin{question}
Does the \textbf{AF}-property imply (1) to (5) of Proposition \ref{pseudometric proof}?
\end{question}
  
  What we know is that the \textbf{AF}-property implies that $\gamma_E^a$ is a pseudometric on $\overline{{\rm Age}_k(E)}^{{\rm BM}}$ and that this pseudometric is complete ((2) and (3) of Lemma \ref{impli}).

\subsection{Ultrapowers of \textbf{AF} Banach spaces}\label{ultrapowers}
Now we proceed to proving that ultrapowers of \textbf{AF} Banach spaces are \textbf{AF} and \textbf{UH}, and obtaining characterizations of $\textbf{AF}$ for ultrapowers.  The \textbf{UH}-property of ultrapowers of $E$ was proven in \cite{ferenczi et all} under the formally stronger assumption that $E$ is Fra\"iss\'e. Recall that for a Banach $E$ and a non-principal ultrafilter $\mathcal U$ on $\mathbb N$, 
$E_\mathcal U$ denotes the ultrapower $E^{\mathbb N}/\mathcal U$.  For $\varepsilon \geq 0$,
We denote by  $(\mathrm{Isom}^{\varepsilon}(E))_\mathcal{U}$
the set of maps $T$ acting on $E_\mathcal U$ by
$T([(x_n)]_{\mathcal U})=[(T_nx_n)]_{\mathcal U}$ for each $[(x_n)]_{\mathcal U}\in E_{\mathcal U}$, where $(T_n)_n$ is a sequence of elements of
$\mathrm{Isom}_{\varepsilon_n}(E)$ with $\lim_n\varepsilon_n=\varepsilon$, and we note that
$(\mathrm{Isom}^\varepsilon(E))_\mathcal{U}
\subseteq
\mathrm{Isom}_{\varepsilon}(E_\mathcal{U})$.

\begin{lemma}\label{AF for subspaces of ultrapowers}
Let $E$ be an \textbf{AF} Banach space and $\mathcal U$ be a non-principal ultrafilter on $\mathbb N$. Then for each $\varepsilon>0$ and $k\in\mathbb N$, there is $\delta>0$ with the following property:
 if  $X\in\overline{\mathrm{Age}_k(E)}^{\mathrm{BM}}$ and $\phi_1,\phi_2\in\mathrm{Emb}_{\delta}(X,E_{\mathcal U})$,
there exists 
$T \in (\mathrm{Isom}^{\varepsilon}(E))_\mathcal{U}$
such that $T\circ\phi_1=\phi_2$.
\end{lemma}

\begin{proof}
  Let $\varepsilon>0$ and $k\in\mathbb N$ be given and $\xi>0$ satisfying $(1+\xi)^2\leq1+\varepsilon$. Fix $\delta_0>0$ corresponding to the definition of \textbf{AF} and set $\delta=\delta_0/2$. Also, let $X\in\overline{\mathrm{Age}_k(E)}^{\mathrm{BM}}$ and $\phi_1,\phi_2\in\mathrm{Emb}_{\delta'}(X,E_\mathcal U)$. Since $\dim X<\infty$, there are two sequences $(\phi_n^1)$ and $(\phi_n^2)$ of linear operators from $X$ to $E$ such that $\phi_1x=[(\phi_n^1(x))]_{\mathcal U}$ and $\phi_2x=[(\phi_n^2(x))]_{\mathcal U}$ for all $x\in X$,  $\|\phi_1\|=\lim_{\mathcal U}\|\phi_n^1\|$ and $\|\phi_2\|=\lim_{\mathcal U}\|\phi_n^2\|$.
  Let $A\in\mathcal U$ be such that $\phi_n^1,\phi_n^2\in\mathrm{Emb}_\delta(X,E)$ for all $n\in A$. By definition of \textbf{AF}, 
 for each $n\in A$, there exists $T_n\in\mathrm{Isom}_{\xi}(E)$ such that $T_n\circ\phi_n^1=\phi_n^2$. Define $T\colon E_{\mathcal U}\to E_{\mathcal U}$ by $[(x_n)]_{\mathcal U}\mapsto[(y_n)]_{\mathcal U}$, where
  \begin{gather*}
      y_n=
      \begin{cases}
      T_n(x_n),& n\in A;\\
      x_n,& n\not\in A.
      \end{cases}
  \end{gather*}
 Suppose that $[(x_n)]_{\mathcal U}=[(x_n')]_{\mathcal U}$, then $\{n\in\mathbb N\,\colon\,\|x_n-x_n'\|<r/(1+\xi)\}\in\mathcal U$ for all $r>0$.
 Thus, $\{n\in A\,\colon\,\|x_n-x_n'\|<r/(1+\xi)\}\in\mathcal U$ and hence, $\{n\in A\,\colon\,\|T_n(x_n)-T_n(x_n')\|<r\}\in\mathcal U$. So, $T$ is well defined. Also, $\|T\|\leq1+\xi$. Now,
 assume that $T([(x_n)]_{\mathcal{ U}})=[(y_n)]_{\mathcal U}={\bf0}.$ Thus, for each $r>0$, $\{n\in\mathbb N\,\colon\,\|y_n\|<r/(1+\xi)\}\in\mathcal U$ and it follows that $\{n\in A\,\colon\,\|y_n\|<r/(1+\xi)\}\in\mathcal U$. Then,
 $\{n\in A\,\colon\,\|x_n\|<r\}\in\mathcal U$, i.e., $[(x_n)]_{\mathcal{U}}={\bf0}.$ From its definition we have $\|T^{-1}\|\leq1+\xi$. Whence $T\in(\mathrm{Isom}^\varepsilon(E))_{\mathcal U}$. Finally, since $A\subset\{n\in\mathbb N\,\colon\,\|\phi_n^2x-(T_n\circ\phi_n^1)x\|<r\}$ for all $r>0$ and $x\in X$, we have $T\circ\phi_1=\phi_2$.
\end{proof}

Note that the previous proof also works  
under the assumption that $E$ is weak \textbf{AF} and that $X \in \mathrm{Age}(E)$. So it is worth noting  the next result.

\begin{lemma}\label{wAF for subspaces of ultrapowers}
Let $E$ be a weak \textbf{AF} Banach space and $\mathcal U$ be a non-principal ultrafilter on $\mathbb N$. Then for each $\varepsilon>0$ and $X\in\mathrm{Age}(E)$, there is $\delta>0$ with the following property:
 if $\phi_1,\phi_2\in\mathrm{Emb}_{\delta}(X,E_{\mathcal U})$,
there exists 
$T \in (\mathrm{Isom}^{\varepsilon}(E))_\mathcal{U}$ such that $T\circ\phi_1=\phi_2$.
\end{lemma}

\begin{lemma}\label{ultrahomogeneity for subspaces}
Let $E$ be an \textbf{AF} Banach space and $\mathcal U$ be a non-principal ultrafilter on $\mathbb N$. Then for every $k\in\mathbb N$, $X\in\overline{\mathrm{Age}_k(E)}^{\mathrm{BM}}$ and  $\phi_1,\phi_2\in\mathrm{Emb}(X,E_{\mathcal U})$, there exists $T\in\mathrm{Isom}(E_{\mathcal U})$ such that $T\circ\phi_1=\phi_2$.
\end{lemma}

\begin{proof}
Let $k\in\mathbb N$ be fixed and also let $X\in\overline{\mathrm{Age}_k(E)}^{\mathrm{BM}}$ and $\phi_1,\phi_2\in\mathrm{Emb}(X,E_\mathcal U)$. For each $m\in\mathbb N$, let $\delta_m>0$ be the corresponding number to $1/m$ in the definition of \textbf{AF} and assume that $\delta_m\to0$. Since $\dim X<\infty$, there are two sequences $(\phi_n^1)$ and $(\phi_n^2)$ of linear operators from $X$ to $E$ and a null sequence of positive numbers $(\alpha_n)$ such that $\phi_1 x=[(\phi_n^1x)]_{\mathcal U}$ and $\phi_2 x=[(\phi_n^2x)]_{\mathcal U}$ for all $x\in X$, and for $i=1,2$ we have $\|\phi_i\|=\lim_{\mathcal U}\|\phi_n^i\|$ and $\phi_n^i$ is an $\alpha_n$-isometry for each $n\in\mathbb N$. We may suppose that $\alpha_m<\delta_m$ for each $m\in\mathbb N$.  So, the set $A_m=\{n\in\mathbb N\,\colon\,\phi_n^1,\phi_n^2\in\mathrm{Emb}_{\delta_m}(X,E)\}$ is in $\mathcal U$ for every $m\in\mathbb N$. By taking small perturbations, we assume that $\phi_n$ is not an isometry for each $n\in\mathbb N$. Thus, $\bigcap_S A_m=\emptyset$ for each  $S\subset \mathbb N$ infinite. By definition of \textbf{AF}, for each $m\in\mathbb N$ and $n\in A_m$, there is $T_n^m\in\mathrm{Isom}_{1/m}(E)$ with $T_n^m\circ\phi_n^1=\phi_n^2$. If $n\in\mathbb N$, let
$k(n)\in\mathbb N$ be the maximal $k$ satisfying $n\in A_k$ and define $T\colon E_{\mathcal U}\to E_{\mathcal U}$ by $[(x_n)]_{\mathcal U}\mapsto[(y_n)]_{\mathcal U}$, where
  \begin{gather*}
      y_n=
      \begin{cases}
      T_n^{k(n)}(x_n),& n\in\bigcup A_m;\\
      x_n,& n\not\in\bigcup A_m.
      \end{cases}
  \end{gather*}
 Since $k(n)\geq n$ for each $n\in\mathbb N$,
 $T$ is an isometry and by following the proof of Lemma \ref{AF for subspaces of ultrapowers} we conclude that $T\circ\phi_1=\phi_2$.
\end{proof}

It is worth noting the following  statement for weak \textbf{AF} spaces, which is obtained through a similar proof.

\begin{lemma}
Let $E$ be a weak \textbf{AF} Banach space and $\mathcal U$ be a non-principal ultrafilter on $\mathbb N$. Then for every $X\in\mathrm{Age}(E)$ and  $\phi_1,\phi_2\in\mathrm{Emb}(X,E_{\mathcal U})$, there exists $T\in\mathrm{Isom}(E_{\mathcal U})$ such that $T\circ\phi_1=\phi_2$.
\end{lemma}

\begin{theorem}\label{ultrapowers of AF are UH}
Let $E$ be a Banach space and $\mathcal U$ be a non-principal ultrafilter on $\mathbb N$. If $E$ is \textbf{AF}, then $E_\mathcal U$ is \textbf{AF} and \textbf{UH}.
\end{theorem}

\begin{proof} Since $\mathrm{Age}_k(E_{\mathcal U})\equiv\overline{\mathrm{Age}_k(E)}^{\mathrm{BM}}$ for each $k\in\mathbb N$ the \textbf{AF}-result follows from Lemma \ref{AF for subspaces of ultrapowers}.
The \textbf{UH} follows from Lemma \ref{ultrahomogeneity for subspaces} and from the
fact that $\mathrm{Age}_k(E_{\mathcal U})\equiv\overline{\mathrm{Age}_k(E)}^{\mathrm{BM}}$ for each $k\in\mathbb N$.
\end{proof}

Recall from the beginning of this subsection that $\mathrm{Isom}^0(E)_{\mathcal U}$ denotes the set
of all elements $T\in\mathrm{Isom}(E_{\mathcal U})$ satisfying 
the next condition: there are a null sequence $(a_n)\in(0,\infty)^{\mathbb N}$  and a sequence of operators $(T_n)$ with  $T_n\in\mathrm{Isom}_{a_n}(E)$ for each $n\in\mathbb N$ such that $T([(x_n)]_{\mathcal U})=[(T_nx_n)]_{\mathcal U}$ for all $[(x_n)]_{\mathcal U}\in E_{\mathcal U}$. The set $\mathrm{Isom}^0(E)_{\mathcal U}$ can be compared with 
the set $\mathrm{Isom}(E)_{\mathcal U}$ of all elements of $\mathrm{Isom}(E_\mathcal{U})$ of the form $[(x_n)]_{\mathcal U}\mapsto[(g_n(x_n))]_{\mathcal U}$ for some sequence $(g_n)\in\mathrm{Isom}(E)^{\mathbb N}$, which was considered in \cite[Proposition 2.15]{ferenczi et all}. We have the obvious inclusion
 $\mathrm{Isom}(E)_{\mathcal U} \subseteq \mathrm{Isom}^0(E)_{\mathcal U}$.
 
\begin{proposition}
\begin{enumerate}
\item if $E_\mathcal{U}$ is \textbf{AUH} and the set $\mathrm{Isom}^0(E)_{\mathcal U}$ is dense with respect to the strong operator topology in $\mathrm{Isom}(E_\mathcal{U})$, then $E$ is \textbf{AF}.
\item if $E_\mathcal{U}$ is almost \textbf{UH} and the set $(\mathrm{Isom}^\delta(E))_{\mathcal U}$ is dense with respect to the strong operator topology in $\mathrm{Isom}_\delta(E_\mathcal{U})$, then $E$ is \textbf{AF}.
\end{enumerate}
\end{proposition}

\begin{proof}
Suppose that $E$ is not \textbf{AF}. Then there are $k_0\in\mathbb N$ and  $0<\varepsilon_0<1/4$ such that for each $n\in\mathbb N$, there exist $X_n\in\mathrm{Age}_{k_0}(E)$ and $\phi_n\in\mathrm{Emb}_{1/n}(X_n,E)$ with $A|_{X_n}\neq\phi_n$ for all $A\in\mathrm{Isom}_{\varepsilon_0}(E)$. Let $\hat{X}$ be the natural finite dimensional subspace of $E_\mathcal U$ associated to the sequence $(X_n)$ and $\phi\colon\hat{X}\to E_\mathcal{U}$ be 
defined as $\phi\hat{(x_n)}=[(\phi_nx_n)]_{\mathcal U}$, where $x_n\in X_n$ for each $n\in\mathbb N$. 

In case (1), since $E_\mathcal{U}$ is \textbf{AUH} and $\phi\in\mathrm{Emb}(\hat{X},E_\mathcal{U})$, there is $T\in\mathrm{Isom}(E_\mathcal{U})$ satisfying $\|T|_{\hat{X}}-\phi\|<\varepsilon_0/4k_0$. By the density of $\mathrm{Isom}^0(E)_{\mathcal U}$, there is $S\in\mathrm{Isom}^0(E)_{\mathcal U}$
such that $\|T-S\|_{\hat{X}}<\varepsilon_0/4k_0$. 
In case (2), since $E_\mathcal{U}$ is almost \textbf{UH} and $\phi\in\mathrm{Emb}(\hat{X},E_\mathcal{U})$, there is $T\in\mathrm{Isom}_{\varepsilon/4k_0}(E_\mathcal{U})$ satisfying $T|_{\hat{X}}=\phi$. By the density of $\mathrm{Isom}^\delta(E)_{\mathcal U}$, there is $S\in\mathrm{Isom}^\delta(E)_{\mathcal U}$
such that $\|T-S\|_{\hat{X}}<\varepsilon_0/4k_0$, where $\delta:=\varepsilon_0/4k_0$.

So $\|\phi-S|_{\hat{X}}\|<\varepsilon_0/2k_0$ which means that $\displaystyle\lim_\mathcal{U}\|\phi_n-S_n|_{X_n}\|<\varepsilon_0/2k_0,$ where $(S_n)$ is a sequence of $a_n$-isometries from $E$ onto $E$ with $a_n\to0$ in Case 1 or
$a_n\to \delta$ in Case 2. Thus 
$A=\{n\in\mathbb N\,\colon\,\|\phi_n-S_n|_{X_n}\|<\varepsilon_0/2k_0\}\in\mathcal U.$ Choose $n\in\mathbb N$ such that $\|\phi_n-S_n|_{X_n}\|<\varepsilon_0/2k_0$
and
$S_n$ is a $\epsilon_0/3$-isometry. If $P_{X_n}\colon E\to X_n$ is a projection with $\|P_{X_n}\|\leq k_0$, then $A_n:=S_n+(\phi_n-S_n)\circ P_{X_n}$ extends $\phi_n$ and a small computation shows that it is a $\varepsilon_0$-isometry, which is absurd. Hence $E$ is
\textbf{AF}.
\end{proof}

This is to compare with \cite[Proposition 2.15]{ferenczi et all} stating (implicitely) that if $E_\mathcal{U}$ is \textbf{AUH} and the set $\mathrm{Isom}(E)_{\mathcal U}$ is dense in $\mathrm{Isom}(E_\mathcal{U})$, then $E$ is Fra\"iss\'e. So for example, under the weaker hypothesis of density of $\mathrm{Isom}^0(E)_{\mathcal U}$, we obtain 
the weaker \textbf{AF} property for $E$.

\section{The Fra\"iss\'e property and $\omega$-categorical spaces}

\subsection{Some pseudometrics on $S_E^n$}

 Fixing a Banach space $E$ and $n\in\mathbb N$, we denote by $(T,x) \mapsto T\cdot x$ the usual (isometric) action of ${\rm Isom}(E)$ on $S_E^n$ defined in \eqref{action on S_E}, i.e. 
 \begin{gather}
    (T,(x_1,\ldots,x_n))\mapsto(Tx_1,\ldots,Tx_n).
\end{gather}
On $E^n$ we consider the $\ell_2$-sum norm. 

\begin{definition}\label{def51}
 If $x,y\in S_E^n$, we set
 \begin{gather*}
  d(x,y)=\inf_{T \in {\rm Isom}(E)}\|Tx-y\|.
 \end{gather*}
 \end{definition}

 If necessary we denote  by $\tilde{d}$ the induced metric on the quotient $Q_n$ of $S_E^n$ by the relation $x\sim y\Longleftrightarrow d(x,y)=0$. From Lemma \ref{pseudometric complete induced by G} we have:

  \begin{fact}\label{completeness of d}
  $d$ is a complete pseudometric. Consequently, $\tilde d$ is a complete metric.
 \end{fact}

On $S_E^n$, by analogy with the Banach-Mazur pseudometric, we also consider the following pseudometric.

\begin{definition}
 If $x,y\in S_E^n$, we set
 \begin{gather*}
  d_{\mathrm{BM}}(x,y)=\log \|A\| \|A^{-1}\|,
 \end{gather*}
if there exists (a necessarily unique) linear invertible map $A$ from $[x]$ to $[y]$ with $A(x_i)=y_i$ for all $i=1,\ldots,n$,
and
$d_{\mathrm{BM}}(x,y)=+\infty$
otherwise. 
\end{definition} 

Let
$\tilde{d}_{\mathrm{BM}}$ denote the distance induced on the quotient of $S_E^n$ by the relation $x\sim_{BM} y\Longleftrightarrow d_{\mathrm{BM}}(x,y)=0$. Finally we also consider a third and less classical pseudometric:

\begin{definition}
 If $x,y\in S_E^n$, we set
 \begin{gather*}
 d^a(x,y)=\sup_{\delta>0}\inf_{T \in {\rm Isom}_\delta(E)}\|Tx-y\|.
 \end{gather*} 
  
 \end{definition}

\begin{lemma}
The function $d^a$ is a pseudometric on $S_E^n$.
\end{lemma}
\begin{proof}
Let us define $\displaystyle d_a^\delta(x,y):=\inf_{T \in {\rm Isom}_\delta(E)}
  \|Tx-y\|$. By using that ${\rm Isom}_\delta(E)$ is invariant under taking inverses, we obtain that
  $(1+\delta)^{-1}d_a^\delta(x,y) \leq d_a^\delta(y,x) \leq
  (1+\delta)d_a^\delta(x,y)$ and therefore that $d_a$ is symmetric. The triangle inequality follows from the  estimate 
  $$d_a^\delta(x,y)+d_a^\delta(z,y) \geq (1+\delta)^{-1} 
  d_a^{\delta^2+2\delta}(x,z),$$
  which we leave to the reader as an exercise.
\end{proof}

 When necessary we denote by $\tilde{d}_a$ the induced distance on the quotient $Q_n^a$ of $S_E^n$ by the relation
 $x \sim^a y \Leftrightarrow d_a(x,y)=0$.
 
 \
 
 We observe the following immediate relations between these pseudometrics and corresponding ultrahomogeneity properties.
\begin{fact}\label{ultrahomogeneity and the metric d_BM}
\
\begin{itemize}
 \item[(i)] If $x,y\in S_E^n$ then $d_a(x,y) \leq d(x,y)$.
 \item[(ii)] If $x,y\in S_E^n$ and $d_a(x,y)=0$, then $d_{\mathrm{BM}}(x,y)=0$.
\item[(iii)] $E$ is \textbf{AUH} if and only if whenever $n\in\mathbb N$ and $x,y\in S_E^n$ satisfy $d_{\mathrm{BM}}(x,y)=0$, we have $d(x,y)=0$.
\item[(iv)] $E$ is almost \textbf{UH} if and only if whenever $n\in\mathbb N$ and $x,y\in S_E^n$ satisfy $d_{\mathrm{BM}}(x,y)=0$, we have $d_a(x,y)=0$.
\end{itemize}
\end{fact}

\begin{notation}
Let ${\mathcal R}=\{R_1,\ldots,R_k\}$ be any list of  relations of linear dependence between the elements of an $n$-uple $(x_1,\ldots,x_n)$ of $S_E^n$ such that none of the relations of the list is consequence of the others.
We denote by $(S_E^n)_{\mathcal R}$ the set of $n$-uples of $S_E^n$ satisfying ${\mathcal R}$ and no additional relation of linear dependence. In particular $(S_E^n)_{\emptyset}$ is the set of $n$-uples of $S_E^n$ which are linearly independent.
\end{notation}

The point of this technical notation is that in order that $d_{\mathrm{BM}}(x,y)<+\infty$, $x$ and $y$ must belong to a same set $(S_E^n)_{\mathcal R}$. So this formalization will help us deal with discontinuities of $d_{\mathrm{BM}}$ with respect to $d$. As an example one may think of a sequence of couples $(x_1,x_k)_k$ with $(x_k)_k$ tending to $x_1$ and $x_k \neq x_1$. Then $(x_1,x_k)_k$  tends to $(x_1,x_1)$ with respect to $d$ but not with respect to $d_{\mathrm{BM}}$.

\begin{fact} \label{conti}
 Given any ${\mathcal R}$ as above, the map
 ${\rm Id}\colon ((S_E^n)_{\mathcal R},d)
 \rightarrow ((S_E^n)_{\mathcal R},d_{\mathrm{BM}})$ is continuous.
\end{fact}

\begin{proof} Let 
$\varepsilon>0$ be given. 
  Fix $(x_i)_{i \in I}$ a basis of $[x_1,\ldots,x_n]$ with constant $K$.  If $d(x,y)<\alpha$, without loss of generality we may assume $\|x-y\| \leq \alpha$. Classical estimates guarantee that if $\alpha$ was small enough, then 
  $(y_i)_{i \in I}$ is a $2K$-basis of $[y_i]$ and that the map $t$ defined by $x_i \mapsto y_i$, $i \in I$, is a $(1+\varepsilon)$-isomorphism.
  Since both $x$ and $y$ belong to $(S_E^n)_{\mathcal R}$, 
  this maps sends $x_i$ to $y_i$ for all the other value of $i$ as well, so
  $d_{\mathrm{BM}}(x,y) \leq \varepsilon$.
\end{proof}

We note that uniform continuity holds if we restrict to a subset where we control the basis constant of $(x_i)_{i \in I}$:
for $K \geq 1$, let $(S_E^n)_K$ be the set of $x \in S_E^n$ which are a basic sequence with constant at most $K$.
 Note that $(S_E^n)_K$ is $d_{\mathrm{BM}}$-closed and therefore $d$-closed. Hence:
 
 \begin{fact}\label{continuity uniform of Id}
 Given $n \in \mathbb N$, given any $K \geq 1$, the map
 ${\rm Id}\colon ((S_E^n)_{K},d)
 \rightarrow ((S_E^n)_{K},d_{\mathrm{BM}})$ is uniformly continuous.
 \end{fact}
 
 For $n$ integer, we use the well-known fact that every $n$-dimensional space has a  basis with constant $\cn$ (since every $n$-dimensional space is $\sqrt{n}$ isomorphic to $\ell_2^n$).
 
 \begin{fact} The following statements are equivalent:
 \begin{itemize}
  \item[(1)]  The space $E$ is weak Fra\"iss\'e.
  \item[(2)] For any $n$, the map ${\rm Id}\colon ((S_E^n)_{2\cn},d)
 \rightarrow ((S_E^n)_{2\cn},d_{\mathrm{BM}})$ is an homeomorphism.
  \item[(3)] For any $n$, for any $K \geq 1$ the map ${\rm Id}\colon ((S_E^n)_{K},d)
 \rightarrow ((S_E^n)_{K},d_{\mathrm{BM}})$ is an homeomorphism.
  \item[(4)] For any $n$, the map ${\rm Id}\colon ((S_E^n)_{\emptyset},d)
 \rightarrow ((S_E^n)_{\emptyset},d_{\mathrm{BM}})$ is an homeomorphism.
 \end{itemize}
 \end{fact}

 \begin{proof}
 $(3) \Rightarrow (2)$ is obvious. $(4) \Rightarrow (3)$ also holds because 
 $(S_E^n)_{\emptyset}=\cup_{K \geq 1} (S_E^n)_K$.
   $(2) \Rightarrow (1)$: 
  Fix $X \in\mathrm{Age}(E)$ and $x=(x_1,\ldots,x_n) \in S_E^n$  a $\cn$-basis of $X$. 
  If $t$ is an $(1+\alpha)$-isometric embedding of $[x]$ into $E$, let
  $y=(y_1,\ldots,y_n)=(tx_1,\ldots,tx_n)$, which belongs to $(S_E^n)_{2\cn}$ if $\alpha$ was chosen small enough.  Then we deduce from $d_{\mathrm{BM}}(x,y)<\alpha$ that $d(x,y)<\varepsilon$, i.e. there exists $T \in {\rm Isom}(E)$ so that 
  $\|Tx-y\| <\varepsilon$,
  so $\|T|_{[x]}-t\|$ is small if $\varepsilon$ was well chosen.
  
  $(1) \Rightarrow (4)$: because of  Fact \ref{conti} for $\emptyset$, we just need to prove that for any $\varepsilon>0$,
   for any $x=(x_1,\ldots,x_n) \in (S_E^n)_{\emptyset}$,
  there exists $\alpha>0$, such that
  $d_{\mathrm{BM}}(x,y)<\alpha \Rightarrow d(x,y)<\varepsilon$.
  If $d_{\mathrm{BM}}(x,y)<\alpha$, i.e. there is $t$ an $(1+\alpha)$-isometric embedding of $[x]$ into $E$ defined by $tx_i=y_i$ and 
  if $\alpha$ was chosen small enough, then  there exists $T \in {\rm Isom}(E)$ so that $\|T|_{[x]}-t\|<\varepsilon$,
  therefore
  $\|Tx_i-y_i\| <\varepsilon$ for all $i$,
  and $\|Tx-y\| \leq n\varepsilon$, so $d(x,y)<n\varepsilon$.
 \end{proof}

 \begin{fact}\label{equivalence of Fraisseness}
 The following assertions are equivalent:
 \begin{itemize}
  \item[(1)]  The space $E$ is  Fra\"iss\'e.
  \item[(2)] For any $n\in\mathbb N$, the map ${\rm Id}\colon ((S_E^n)_{2\cn},d)
 \rightarrow ((S_E^n)_{2\cn},d_{\mathrm{BM}})$ is a uniform homeomorphism.
  \item[(3)] For any $n\in\mathbb N$, for any $K \geq 1$, the map ${\rm Id}\colon ((S_E^n)_{K},d)
 \rightarrow ((S_E^n)_{K},d_{\mathrm{BM}})$ is a uniform homeomorphism.
 \item[(4)] For any $n\in\mathbb N$, the map ${\rm Id}\colon ((S_E^n)_{\emptyset},d)
 \rightarrow ((S_E^n)_{\emptyset},d_{\mathrm{BM}})$ is a homeomorphism with uniformly continuous inverse.
 \end{itemize}
 \end{fact}

 \begin{proof} $(3) \Rightarrow (2)$ is obvious. $(4) \Rightarrow (3)$  holds because of Fact \ref{continuity uniform of Id}  and because
 $(S_E^n)_{\emptyset}=\cup_{K \geq 1} (S_E^n)_K$.
   $(2) \Rightarrow (1)$: 
  Fix $X \in\mathrm{Age}(E)$ and $x=(x_1,\ldots,x_n) \in S_E^n$  a $\cn$-basis of $X$. 
  If $t$ is an $(1+\alpha)$-isometric embedding of $[x]$ into $E$, let
  $y=(y_1,\ldots,y_n)=(tx_1,\ldots,tx_n)$. Then $y$ is a $2\cn$-basis if $\alpha$ was chosen small enough. Then we deduce from $d_{\mathrm{BM}}(x,y)<\alpha$ that $d(x,y)<\varepsilon$, i.e. there exists $T \in {\rm Isom}(E)$ so that 
  $\|Tx-y\| <\varepsilon$,
  so $\|T|_{[x]}-t\|$ is small if $\varepsilon$ was well chosen.

  $(1) \Rightarrow (2)$: because of Fact \ref{conti} for $\emptyset$, we just need to prove that for any $n$ and any $\varepsilon>0$,
  there exists $\alpha>0$, such that 
   for any $x,y \in (S_E^n)_{\emptyset}$,
  $d_{\mathrm{BM}}(x,y)<\alpha \Rightarrow d(x,y)<\varepsilon$. 
  Fix  some $n$, and fix $\varepsilon>0$.
  If $d_{\mathrm{BM}}(x,y)<\alpha$, i.e. there is $t$ an $(1+\alpha)$-isometric embedding of $[x]$ into $E$ defined by $tx_i=y_i$, 
  if $\alpha$ was chosen small enough then  there exists $T \in {\rm Isom}(E)$ so that $\|T|_{[x]}-t\|<\varepsilon$,
  therefore
  $\|Tx_i-y_i\| <\varepsilon$ for all $i$,
  and $\|Tx-y\| \leq n\varepsilon$, so $d(x,y)<n\varepsilon$.
 \end{proof}
 
 \subsection{On $\omega$-categorical spaces}\label{omega-categoricity}
 
We finally recall the notion of $\omega$-categoricity of Banach spaces, see \cite{ben-yaacov I} for a very general study. Recall the definition of the pseudometric $d$ from Definition \ref{def51}.

\begin{definition}
 A Banach space $E$ is $\omega$-categorical when the quotient $Q_n$ of $S_E^n$ by the orbit relation of the action of ${\rm Isom}(E)$ is $\tilde{d}$-compact. 
\end{definition}

All $L_p(0,1)$-spaces, $1 \leq p<+\infty$, and the Gurarij space $\mathbb G$ are examples of $\omega$-categorical \cite[Section 17]{ben-yaacov I} and \cite[Section 2]{ben-yaacov II}, respectively. Moreover, $\omega$-categorical Banach spaces contains isometric copies of $c_0$ or $\ell_p$, $1\leq p<\infty$ \cite[Corollary 5.13]{khanaki}. 

\begin{theorem}\label{characterization of fraisseness}
 A  Banach space is Fra\"iss\'e if and only if it is
 \textbf{AUH} and $\omega$-categorical. 
\end{theorem}

\begin{proof} Fix $K \geq 1$. Fix a \textbf{AUH} space $E$,
 and consider the map
 $${\rm Id}\colon ((S_E^n)_K,\tilde{d}) \rightarrow ((S_E^n)_K,\tilde{d}_{\mathrm{BM}}).$$
 By Facts \ref{ultrahomogeneity and the metric d_BM} and \ref{continuity uniform of Id} it defines a (uniformly) continuous bijective map between metric spaces.
 
 If $E$ is $\omega$-categorical
 and since $(S_E^n)_K$ is $\tilde{d}$-closed in $S_E^n$, the domain of this ${\rm Id}$-map is compact. Therefore ${\rm Id}$ is a uniform homeomorphism.  We conclude by Fact \ref{equivalence of Fraisseness}(3).
 
 Conversely assume $E$ is Fra\"iss\'e. Let $y_k=(x_1^k,\ldots,x_n^k)$ be a sequence in
 $S_E^n$. We describe how to find a $\tilde{d}$-converging subsequence of $y_k$
 
 We first look at $[x_1^k,x_2^k]$. Passing to a subsequence
 either this is basic with a fixed constant $K_2$, or $d(x_2^k,\mathbb{R} x_1^k)$ tends to $0$. In the second case, we may find $\lambda$ 
 such that $d(x_2^k, \lambda x_1^k)$ tends to $0$, and therefore assume wlog that $x_2^k=\lambda x_1^k$. Repeating this at each step  we obtain $I \subseteq \{1,\ldots,n\}$, $K$ and ${\mathcal R}$ such that we may assume wlog that for all $k$
 \begin{itemize}
     \item $\{x_i^k\,\colon\, i \in I\}$ is a  $K$-basic sequence
     \item $[y_k]=[x_i^k\,\colon\, i \in I]$
     \item $\{x_1^k, \ldots, x_n^k\} \in (S_E^n)_{\mathcal R}$.
 \end{itemize}
 Without loss of generality we may assume that $\{x_i^k\,\colon\, i \in I\}$ is $d_{\mathrm{BM}}$-convergent and therefore $\{x_1^k,\ldots,x_n^k\}$ as well. The $d_{\mathrm{BM}}$-limit  $\{z_1,\ldots,z_n\}$ of $\{x_1^k,\ldots,x_n^k\}$
 will satisfy the three conditions that
  $\{z_i\,\colon\, i \in I\}$ is a  $K$-basic sequence, $[z_i\,\colon\, 1\leq i\leq n]=[z_i\,\colon\, i \in I]$, and
     $\{z_1, \ldots, z_n\} \in (S_E^n)_{\mathcal R}$. Now, from
     Fact \ref{equivalence of Fraisseness}, it follows that $\{x_1^k,\ldots,x_n^k\}$ is $d$-Cauchy, and hence $d$-convergent in $S_E^n$ by Fact \ref{completeness of d}. So,  the limit has to coincide with $\{z_1,\ldots,z_n\}$.  In the end we have that $\{x_i^k\,\colon\, i \in I\}$ is $d$-convergent to $\{z_i\,\colon\, i \in I\}$. Since they all belong to
 $(S_E^n)_{\mathcal R}$, we finally have $d$-convergence of
 $y_k$ to $\{z_1,\ldots,z_n\}$.
\end{proof}

$L_p$ spaces for $p=4,6,\ldots$ are examples of $\omega$-categorical spaces which are not Fra\"iss\'e. The proof from \cite[Chapter 4]{ferenczi et all} that for $p 
\notin 2\mathbb N+4$, $L_p[0,1]$ is Fra\"iss\'e, relies on technical estimates regarding approximate versions of the equimeasurability theorem of Plotkin and Rudin. In some sense from that proof one could hope to obtain some explicit estimates regarding the parameters $\varepsilon$ and $\delta$ of the Fra\"iss\'e property. However as Ben Yaacov says \cite{ben-yaacov III},  one can use Theorem \ref{characterization of fraisseness}
to obtain an abstract, model theoreric proof of the result of \cite{ferenczi et all}. Just combine the result of Lusky that $L_p(0,1)$ is approximately ultrahomogeneous for $p 
\notin 2\mathbb N+4$, with the $\omega$-categorical property of $L_p$-spaces. 

\section{Final remarks and Questions}

In this section we propose some questions that emerged throughout this work. Answering those questions could in our view give a better understanding of the theory.

\begin{question}
    Is the $\ell_\infty$-sum of \textbf{FIE} (\textbf{aFIE}) Banach spaces a \textbf{FIE} (\textbf{aFIE}, respectively) Banach space?
\end{question}

Example \ref{FIE-ness of sums of c_0(a)} says that answer is yes for the \textbf{FIE}-property of arbitrary $\ell_\infty$-sums of $c_0(\Gamma)$.

\begin{question}
Are the properties \textbf{aFIE} and \textbf{FIE} equivalent?
\end{question}

We know that answer is yes for reflexive spaces (see (2) of Proposition \ref{afie-ness y complementos}).

\vspace{1.5mm}

Corollary \ref{AFandwAF} motivates the next two questions.

\begin{question}
   When does the \textbf{AF}-property imply that 
$\mathrm{Age}_k(E)$ is compact for all $k$? 
\end{question}

\begin{question}
    When does the weak \textbf{AF} property imply the \textbf{AF} property?
\end{question}

Another natural question is:

\begin{question}
If $X$ and $Y$ are almost Fra\"iss\'e and almost isometric, must they be isometric?
\end{question}

In \cite[Problem 2.9]{ferenczi et all}, it is asked what other separable spaces different from $\mathbb G$ and $L_p(0,1)$, $p\not\in 2\mathbb N+4$, are Fra\"iss\'e. Another related question is the following:

\begin{question}
    Are $\mathbb G$, $L_p(0,1)$ for $p\not\in 2\mathbb N+4$ and the Hilbert space $\ell_2$ the only separable almost ultrahomogeneous Banach spaces?
\end{question}

The notion of $\omega$-categoricity could suggest defining and studying the next weaker property:

\begin{definition}
 A Banach space $E$ is \textit{almost $\omega$-categorical} if the quotient $Q_n$ of $S_E^n$ by the orbit relation of the action of ${\rm Isom}(E)$ is $\tilde{d_a}$-compact. 
\end{definition}

Since $d_a \leq d$, the identity map on $S_E^n$ with respect to $d$ and $d_a$ is continuous and therefore any $\omega$-categorical space is also almost $\omega$-categorical. Inspired by Theorem \ref{characterization of fraisseness} we ask:

\begin{question}
Is a Banach space almost Fra\"iss\'e if and only if it is
 almost ultrahomogeneous and almost $\omega$-categorical?
\end{question}

 There seem to be technical difficulties with lack of completeness of $d_a$. We conjecture that answer is yes if we add as hypothesis that the age of the space is BM-closed.

\section*{Acknowledgments}

 The research of this paper was developed during a postdoctoral stay of the second author supported by Funda\c c\~ao de Apoio \`a Pesquisa do Estado de S\~ao Paulo, FAPESP, Processos 2016/25574-8 and 2021/01144-2, and UIS.

\end{document}